\documentclass[12pt,english]{article}
\usepackage[T1]{fontenc}
\usepackage[utf8]{inputenc}
\usepackage{geometry}
\usepackage[skip=\smallskipamount]{parskip}
\usepackage{babel}
\usepackage{mathtools}
\usepackage{amsmath}
\usepackage{amsthm}
\usepackage{amssymb}
\usepackage[bookmarks=true,bookmarksnumbered=false,bookmarksopen=false,
 breaklinks=false,pdfborder={0 0 0},pdfborderstyle={},backref=false,colorlinks=false]
 {hyperref}
\hypersetup{
 pdfauthor={Marco Praderio Bova}}

\makeatletter
\theoremstyle{plain}
\newtheorem*{conjecture*}{\protect\conjecturename}
\newcommand\thmsname{\protect\theoremname}
\newcommand\nm@thmtype{theorem}
\theoremstyle{plain}

\newenvironment{namedthm}[1][Undefined Theorem Name]{
  \ifx{#1}{Undefined Theorem Name}\renewcommand\nm@thmtype{theorem*}
  \else\renewcommand\thmsname{#1}\renewcommand\nm@thmtype{namedtheorem}
  \fi
  \begin{\nm@thmtype}}
  {\end{\nm@thmtype}}
\theoremstyle{remark}
\newtheorem*{acknowledgement*}{\protect\acknowledgementname}
\theoremstyle{plain}
\newtheorem{thm}{\protect\theoremname}[section]
\theoremstyle{remark}
\newtheorem{notation}[thm]{\protect\notationname}
\theoremstyle{definition}
\newtheorem{example}[thm]{\protect\examplename}
\theoremstyle{remark}
\newtheorem{rem}[thm]{\protect\remarkname}
\theoremstyle{definition}
\newtheorem{defn}[thm]{\protect\definitionname}
\theoremstyle{plain}
\newtheorem{prop}[thm]{\protect\propositionname}
\theoremstyle{plain}
\newtheorem{lem}[thm]{\protect\lemmaname}
\theoremstyle{plain}
\newtheorem{cor}[thm]{\protect\corollaryname}

\newcommand{\Orbitize}[1]{\mathcal{O}\left(#1\right)}

\newcommand{\bjarrow}{
  \hookrightarrow\mathrel{\mspace{-15mu}}\rightarrow
}

\date{}

\newcommand{\Hom}{\operatorname{Hom}}
\newcommand{\limn}[1][n]{\operatorname{lim}^{#1}}
\newcommand{\Ext}{\operatorname{Ext}}
\newcommand{\Aut}{\operatorname{Aut}}
\newcommand{\Out}{\operatorname{Out}}
\newcommand{\Id}{\operatorname{Id}}
\newcommand{\im}{\operatorname{im}}

\newcommand{\Tot}{\operatorname{Tot}}
\newcommand{\Syl}{\operatorname{Syl}}
\newcommand{\Sol}{\operatorname{Sol}}
\newcommand{\Spin}{\operatorname{Spin}}
\newcommand{\SSets}{\operatorname{SSets}}
\newcommand{\Top}{\operatorname{Top}}
\newcommand{\Tr}{\operatorname{Tr}}
\newcommand{\Mod}{\operatorname{-Mod}}
\renewcommand{\mod}{\operatorname{-mod}}
\newcommand{\hocolim}{\operatorname{hocolim}}

\newcommand{\lui}[2]{\prescript{#1}{}{#2}}

\newcommand{\R}{\mathcal{R}}
\newcommand{\RR}{\mathcal{S}}

\newcommand{\Z}{\mathbb{Z}}
\newcommand{\Zp}[1][p]{\Z_{#1}}
\newcommand{\Fp}[1][p]{\mathbb{F}_{#1}}

\newcommand{\F}{\mathcal{F}}
\renewcommand{\L}{\mathcal{L}}
\newcommand{\OF}[1][\F]{\Orbitize{#1}}
\newcommand{\OFc}[1][\F]{\Orbitize{#1^c}}
\newcommand{\OFD}[2]{\mathcal{O}_{#2}\left(#1\right)}
\newcommand{\OFC}[1][\F]{\OFD{#1}{\mathcal{C}}}

\newcommand{\FAB}[2]{\F_{#1}\left(#2\right)}
\newcommand{\FSG}[1][G]{\FAB{S}{#1}}
\newcommand{\FS}[1][S]{\FAB{#1}{#1}}

\newcommand{\OGC}[1][G]{\OFD{#1}{\overline{\mathcal{C}}}}

\newcommand{\CHpD}[3]{C_{#1}^{#2'}\left(\mathcal{#3}\right)}
\newcommand{\CGpC}[1][C]{\CHpD{G}{p}{#1}}
\newcommand{\HGnHR}[3]{\mathcal{H}_{#1,#3}^{G,#2}}

\newcommand{\Dn}[1][n]{\mathcal{D}_{#1}}
\newcommand{\Dt}{\Dn[2]}

\newcommand{\PCGHP}[4]{\mathcal{P}_{\mathcal{#3}^{#4}} \left( #2 ; #1 \right)}
\newcommand{\PCGH}[3]{\PCGHP{#1}{#2}{#3}{ }}

\newcommand{\NCGH}[3]{\mathcal{N}_{\mathcal{#3}} \left( #2 ; #1 \right)}

\newcommand{\NNCGH}[3]{\overline{\mathcal{N}}_{\mathcal{#3}} \left( #2 ; #1 \right)}

\newcommand{\TCGH}[3]{\Gamma_{\mathcal{#3}} \left( #2 ; #1 \right)}

\newcommand{\TTCGH}[3]{\overline{\Gamma}_{\mathcal{#3}} \left( #2 ; #1 \right)}

\makeatother

\usepackage[style=numeric,doi=false,isbn=false,url=false,eprint=false, date=year]{biblatex}
\providecommand{\acknowledgementname}{Acknowledgement}
\providecommand{\conjecturename}{Conjecture}
\providecommand{\corollaryname}{Corollary}
\providecommand{\definitionname}{Definition}
\providecommand{\examplename}{Example}
\providecommand{\lemmaname}{Lemma}
\providecommand{\notationname}{Notation}
\providecommand{\propositionname}{Proposition}
\providecommand{\remarkname}{Remark}
\providecommand{\theoremname}{Theorem}

\begin{document}
\title{Higher limits over the fusion orbit category via centralizers of amalgams}
\author{Marco Praderio Bova}
\maketitle
\begin{abstract}
We study the Díaz-Park sharpness conjecture (see \cite{diaz2014mackey})
for fusion systems and prove that, under certain circumstances, there
exists a $4$ terms exact sequence relating the first two higher limits
of the contravariant part of a Mackey functor over certain fusion
systems. We show how this result can be applied to the family of Benson-Solomon
fusion systems thus providing another approach to studying the sharpness
for this family of fusion systems.
\end{abstract}

\section{\protect\label{sec:Introduction.}Introduction.}

Let $p$ be a prime, let $G$ be a finite group, let $\boldsymbol{C}$
be a small category and let $F:\boldsymbol{C}\to\text{Top}$ be a
functor such that for every $X\in\boldsymbol{C}$ the topological
space $F\left(X\right)$ has the homotopy type of the classifying
space of a subgroup of $G$. A $\text{mod }p$ homology decomposition
of $BG$ is a $\text{mod }p$ homology equivalence of the form
\[
\hocolim_{\boldsymbol{C}}\left(F\right)\underset{p}{\tilde{\to}}BG.
\]
Due to a result of Bousfield and Kan (see \cite[\S XII  4.5]{BousfieldKanHomotopyLimitsColimitsCompletionsAndLocalizations})
we know that any $\text{mod }p$ homology decomposition leads to a
first quadrant cohomology spectral sequence
\begin{equation}
\limn[i]_{\boldsymbol{C}}H^{j}\left(F\left(?\right),\Fp\right)\Rightarrow H^{i+j}\left(G,\Fp\right).\label{eq:spectral-sequence-homology-decomposition.}
\end{equation}

In \cite[Theorem 1.6 and Example 1.17]{DwyerHomologyDecomposition}
Dwyer proves that an homology decomposition of $G$ can be obtained
by taking $\boldsymbol{C}$ to be the $p$-centric orbit category
of $G$ (denoted $\mathcal{O}_{p}^{c}\left(G\right)$) and taking
$F$ such that, for every $p$-centric $H\le G$ the topological space
$F\left(H\right)$ has the homotopy type of $BH$. Building on this
Dwyer proves in \cite[Theorem 10.3]{dwyer1998sharp} that the spectral
sequence of Equation (\ref{eq:spectral-sequence-homology-decomposition.})
deriving from such homology decomposition is in fact sharp (i.e. $\limn[i]_{\mathcal{O}_{p}^{c}\left(G\right)}\left(H^{j}\left(?,\Fp\right)\right)=0$
for every $i\ge1$). It follows that for every $n\ge0$ the isomorphism
of abelian groups $\lim_{\mathcal{O}_{p}^{c}\left(G\right)}\left(H^{n}\left(?,\Fp\right)\right)\cong H^{n}\left(G,\Fp\right)$
holds.

On the other hand, work of Broto, Levi and Oliver (see \cite{BrotoLeviOliverHomotopyTheoryOfFusionSystems})
and of Chermak (see \cite{ChermakExistenceLinkingSystem}) leads to
the description, existence and uniqueness of the classifying space
$B\F$ of any fusion system $\F$. From \cite[Proposition 2.2]{BrotoLeviOliverHomotopyTheoryOfFusionSystems}
we know that this classifying space satisfies the weak equivalence
\begin{equation}
B\F\simeq\hocolim_{\OFc}\left(\tilde{B}\left(?\right)\right).\label{eq:classifying-space-F.S.}
\end{equation}
Where $\tilde{B}:\OFc\to\text{Top}$ is the (unique up to equivalence)
homotopy lifting of the classifying space functor $B:\OFc\to\text{hoTop}$.
From \cite[Lemma 5.3]{BrotoLeviOliverHomotopyTheoryOfFusionSystems}
we also know that $\lim_{\OFc}\left(H^{n}\left(?,\Fp\right)\right)\cong H^{n}\left(B\F,\Fp\right)$.
This result, analogous to that obtained in the case of finite groups,
lead Aschbacher, Kessar and Oliver to conjecture in \cite[\S III.7]{FusionSystemsinAlgebraAndTopologyAschbacherKessarOliver2011}
that the higher limits of the cohomology functor over the centric
fusion orbit category vanish. Previous work of Jackowski and McClure
(see \cite[Corollary 5.16]{JackowskiMcClureHomotopyDecompositionViaAbelianSubgroups})
hints in fact at the existence of a more general result which is conjectured
by Díaz and Park in \cite{diaz2014mackey}.
\begin{conjecture*}[Sharpness for fusion systems]
\label{conj:Sharpness-for-fusion-systems}Let $S$ be a finite $p$-group,
let $\F$ be a fusion system over $S$ and let $M=\left(M_{*},M^{*}\right)$
be a Mackey functor over $\F$ on $\Fp$ (see Definition \ref{def:Mackey-functor.}).
Then $\limn_{\OFc}\left(M^{*}\downarrow_{\OFc}^{\OF}\right)=0$ for
every $n\ge1$.
\end{conjecture*}
This conjecture remains unresolved and has recently been the object
of much research (see \cite{CohomologyOnTheCentricOrbitCategoryOfAFusionSystem,GrazianMarmoSharpnessOfFSOnSylowG2p,puncturedGroupsForExoticFusionSystems,HigherLimitsOverTheFusionOrbitCategory}).
In this paper we aim to add to these results by proving the following.\phantomsection
\label{thm:A}
\begin{namedthm}[Theorem A]
Let $S$ be a finite $p$-group, let $G_{1},G_{2}$ be finite groups
such that $S\in\Syl_{p}\left(G_{i}\right)$ for $i=1,2$, let $G:=G_{1}*_{S}G_{2}$,
let $\F:=\FSG$, let $\F_{i}:=\FSG[G_{i}]$ for $i=1,2$, let $\F_{3}:=\FS$
and let $\mathcal{C}$ be the family of all $\F$-centric subgroups
of $S$. For every $\Fp\OFc$-module $M$ such that $\limn_{\OFC[\F_{i}]}\left(M\downarrow_{\OFC[\F_{i}]}^{\OF}\right)=0$
for every $n\ge1$ and $i=1,2,3$ there is an isomorphism
\[
\Ext_{\Fp\OFc}^{n}\left(\CGpC,M\right)\cong\limn[n+2]_{\OFc}\left(M\downarrow_{\OFc}^{\OF}\right)
\]
Here $\CGpC$ is the $\Fp\OFc$-module of which sends every $P\in\OFc$
to $\CGpC\left(P\right)=\text{Ab}\left(C_{G}\left(P\right)/Z\left(P\right)\right)\otimes\Fp$
(see Definition \ref{def:CGpC}) where $\text{Ab}\left(G'\right)$
denotes the abelianization of the group $G'$. Moreover there is a
short exact sequence of the form 
\[
0\to\underset{\OFc}{\limn[1]}\left(M'\right)\to M\left(S\right)/\left(M^{\F_{1}}+M^{\F_{2}}\right)\to\underset{\Fp\OFc}{\Hom}\left(\CGpC,M'\right)\to\underset{\OFc}{\limn[2]}\left(M'\right)\to0.
\]
Where $M':=M\downarrow_{\OFC}^{\OF}$ and the $\Fp$-modules $M^{\F_{i}}:=\lim_{\F_{i}}\left(M\right)$
are seen as subgroups of $M\left(S\right)$ (see Lemma \ref{lem:MF-as-subgroup-of-MS.}).
\end{namedthm}
In Section \ref{sec:last-section} we argue that Theorem \hyperref[thm:A]{A}
might prove useful for studying the sharpness conjecture for the family
of Benson-Solomon fusion systems (see \cite{GroupTheoreticApproachFamilyOf2LocalFiniteGroups,BensonSolomonFSCorrection}).
In particular, using the construction of the Benson-Solomon fusion
systems provided in \cite{GroupTheoreticApproachFamilyOf2LocalFiniteGroups}
we prove the following.

\phantomsection
\label{thm:B}
\begin{namedthm}[Theorem B]
Let $q\equiv\pm3\mod 8$ be a prime, let $n\in\mathbb{N}$, let $H_{\psi_{n}},B_{\psi_{n}},K_{\psi_{n}}$
and $G_{\psi_{n}}$ as above and let $\theta_{\psi_{n}}$ be the signalizer
functor associated to $G_{\psi_{n}}$ (see \cite[Theorem A(4)]{GroupTheoreticApproachFamilyOf2LocalFiniteGroups}):
\begin{enumerate}
\item \label{enu:Theorem-B-1-1}Theorem \hyperref[thm:A]{A} holds with
$p:=2$, $S\in\Syl_{2}\left(\Spin_{7}\left(q^{2^{n}}\right)\right)$,
$G_{1}:=H_{\psi_{n}}$, $G_{2}:=K_{\psi_{n}}$, $\F:=\F_{\Sol}\left(q^{2^{n}}\right)$
(i.e. the Benson-Solomon fusion system over $S$) and $M$ the contravariant
part of a Mackey functor over $\F$ with coefficients in $\Fp[2]$.
\item \label{enu:Theorem-B-2-1}For every $\F_{\Sol}\left(q^{2^{n}}\right)$-centric
subgroup $P$ of $S$ the group $\CGpC\left(P\right)$ is the tensor
product of $\Fp\left[2\right]$ with the abelianization of an extension
of $\theta_{\psi_{n}}\left(P\right)$.
\end{enumerate}
\end{namedthm}
The paper is organized as follows:

In Section \ref{sec:Preliminaries.} we briefly recall the definitions
of orbit category of a fusion system and Mackey functor over a fusion
system and establish most of the notation used throughout this paper.

In Section \ref{sec:Dwyer-spaces} we introduce certain simplexes
whose nature is reminiscent to that of the Dwyer spaces for the normalizer
decomposition (see \cite[Section 3.3]{DwyerHomologyDecomposition})
and study the homology groups of related simplicial sets.

Section \ref{sec:main-section} constitutes the core of this paper
and in it we study the higher limits of a certain poset category $\Dn$
(see Definition \ref{def:Dn}) and use techniques similar to those
employed in \cite{HigherLimitsOverTheFusionOrbitCategory} in order
to relate such higher limits with the simplicial sets introduced in
Section \ref{sec:Dwyer-spaces}. As a result of this process we prove
Theorem \hyperref[thm:A]{A}.

We conclude with Section \ref{sec:last-section} where we briefly
describe how Theorem \hyperref[thm:A]{A} may be used in order to
study the Benson-Solomon fusion systems. In particular, in this section,
we prove Theorem \hyperref[thm:B]{B}.

Appendix \ref{app:Higher-limits-over-Fp-and-Zp.} is included for
the sake of completion and in it we prove that, given a small category
$\boldsymbol{C}$, a functor $F:\boldsymbol{C}\to\Fp\Mod$ and denoting
by $\iota:\Fp\Mod\hookrightarrow\Zp\Mod$ the natural inclusion of
categories, there exists an isomorphism $\iota\left(\limn_{\boldsymbol{C}}\left(F\right)\right)\cong\limn_{\boldsymbol{C}}\left(\iota\circ F\right)$.
This seems to be a well known result which is extensively used in
the literature but for which we failed to find a reference.
\begin{acknowledgement*}
The author would like to thank Lancaster University, the Oberwolfach
Institute of Mathematics and TU Dresden for their support.
\end{acknowledgement*}

\section{\protect\label{sec:Preliminaries.}Preliminaries.}

During this section we briefly recall the definitions of orbit category
of a fusion system (see Definition \ref{def:orbit-category.}) and
Mackey functor over a fusion system (see Definition \ref{def:Mackey-functor.})
and introduce most of the notation that we use throughout this paper.
This is not meant to be an exhaustive introduction and we refer the
interested reader to \cite{IntroductionToFusionSystemsLinckelmann}
and \cite{GuideToMackeyFunctorsWebb} for introductory guides to fusion
systems and Mackey functors respectively.

Let us start by introducing some of the common notation that we use
throughout this paper.
\begin{notation}
\label{nota:initial-notation.}$\phantom{.}$
\begin{itemize}
\item As is usual the letter $p$ denotes a prime number.
\item We refer to saturated fusion systems simply as fusion systems.
\item We refer to rings with unit simply as rings.
\item Let $\R$ be a ring. We refer to right $\R$-modules simply as $\R$-modules
and denote by $\R\Mod$ and $\R\mod{}$ the categories of $\R$-modules
and finitely generated $\R$-modules respectively.
\item We denote by $\Top$ the category of topological spaces, by $\SSets$
the category of simplicial sets and by $\left|\cdot\right|:\SSets\to\Top$
the geometric realization functor.
\item Let $G$ be a group and let $H$ and $K$ be subgroups of $G$. We
denote by $\iota_{H}^{G}$ the natural inclusion from $H$ to $G$,
by $N_{G}\left(H,K\right)$ the set $\left\{ x\in G\,:\,\lui{x}{H}\le K\right\} $
and by $\Hom_{G}\left(H,K\right)$ the set of group morphisms $\left\{ c_{x}:H\to K\,:\,x\in N_{G}\left(H,G\right)\right\} $
where $c_{x}\left(y\right)=xyx^{-1}$ for every $y\in H$ and every
$x\in G$.
\item Let $G$ be a group and let $H$ be a subgroup of $G$. We denote
by $\Aut_{G}\left(H\right)$ the group $\Hom_{G}\left(H,H\right)$
and by $\Out_{G}\left(H\right)$ the quotient $\Aut_{G}\left(H\right)/\Aut_{H}\left(H\right)$.
\item Let $\boldsymbol{C}$ be a category. We write $X\in\boldsymbol{C}$
to denote that $X$ is an object in $\boldsymbol{C}$.
\item Let $\boldsymbol{C}$ be a small category, let $A,B$ and $M$ be
objects in $\boldsymbol{C}$ and let $f\in\Hom_{\boldsymbol{C}}\left(A,B\right)$.
We denote by $f^{*}:\Hom_{\boldsymbol{C}}\left(B,M\right)\to\Hom_{\boldsymbol{C}}\left(A,M\right)$
(resp. $f_{*}:\Hom_{\boldsymbol{C}}\left(M,A\right)\to\Hom_{\boldsymbol{C}}\left(M,B\right)$)
the set maps (or morphisms of abelian groups if $\boldsymbol{C}$
is abelian) sending every morphism $g\in\Hom_{\boldsymbol{C}}\left(B,M\right)$
(resp. $g\in\Hom_{\boldsymbol{C}}\left(M,A\right)$) to the composition
$gf\in\Hom_{\boldsymbol{C}}\left(A,M\right)$ (resp. $fg\in\Hom_{\boldsymbol{C}}\left(M,B\right)$).
\end{itemize}
\end{notation}

\medskip{}

\textbf{Fusion systems} were first devised by Puig in \cite{FrobeniusCategoriesPuig}
as a common framework between $p$-fusion of finite groups and $p$-blocks
of finite groups. Intuitively they can be thought of as categories
that collect the $p$-local structure of a finite group. The most
common example of fusion system is the following.
\begin{example}
\label{exa:fusion-system.}Let $G$ be a finite group and let $S\in\Syl_{p}\left(G\right)$.
The \textbf{fusion system $\FSG$ over $S$} is the category whose
objects are subgroups of $S$ and whose morphisms are given by conjugation
by elements of $G$ (i.e. $\Hom_{\FSG}\left(P,Q\right)=\Hom_{G}\left(P,Q\right)$).
\end{example}

\begin{rem}
We know from \cite{LearyStancuRealisingFusionSystems,RobinsonAmalgams}
that every fusion system (saturated or not) can be written as $\FSG$
for some, non necessarily finite, group $G$ and some finite $p$-group
$S\le G$.
\end{rem}

We are mostly concerned with the orbit category of a fusion system
and certain full subcategories of it.
\begin{defn}
\label{def:orbit-category.}Let $S$ be a finite $p$-group and let
$\F$ be a (non necessarily saturated) fusion system over $S$. The
\textbf{orbit category of $\F$} is the category $\OF$ whose objects
are subgroups of $S$ and whose morphism sets are given by 
\[
\Hom_{\OF}\left(P,Q\right):=\Aut_{Q}\left(Q\right)\backslash\Hom_{\F}\left(P,Q\right)
\]
where $\Aut_{Q}\left(Q\right)$ acts on $\Hom_{\F}\left(P,Q\right)$
by left composition.
\end{defn}

\begin{notation}
Given a fusion system $\F$ and objects $P,Q\in\F$ we write $\overline{\varphi}\in\Hom_{\OF}\left(P,Q\right)$
to indicate that there exists $\varphi\in\Hom_{\F}\left(P,Q\right)$
such that $\varphi$ is a representative of $\,\overline{\varphi}$.
\end{notation}

\begin{defn}
\label{def:OFC.}Let $S$ be a finite $p$-group, let $\F$ be a (non
necessarily saturated) fusion system over $S$ and let $\mathcal{C}$
be a family of subgroups of $S$. We denote by $\OFC$ the full subcategory
of $\OF$ having as objects the elements of $\mathcal{C}$. We define
the \textbf{centric orbit category of $\F$ }by setting $\OFc:=\OFC$
where $\mathcal{C}$ is the family of all $\F$-centric subgroups
of $S$.
\end{defn}

\medskip{}

A \textbf{Mackey functor} is an algebraic structure with operations
that resemble the induction, restriction and conjugation maps in representation
theory. There exist several equivalent definitions of Mackey functors
in a variety of different contexts. In this paper we work with Díaz
and Park's adaptation to fusion systems of Dress' definition of Mackey
functor (see \cite{NotesOnTheTheoryOfRepresentationsOfFiniteGroups}).
\begin{defn}[{\cite[Definition 2.1]{diaz2014mackey}}]
\label{def:Mackey-functor.}Let $S$ be a finite $p$-group, let
$\F$ be a fusion system over $S$, and let $\R$ be a commutative
ring. A \textbf{Mackey functor over $\F$ with coefficients in $\R$}
is a pair $M=\left(M_{*},M^{*}\right)$ of a covariant functor $M_{*}:\OF\to\R\mod{}$
and a contravariant functor $M^{*}:\OF^{\text{op}}\to\R\mod{}$ such
that:
\begin{enumerate}
\item $M\left(P\right):=M_{*}\left(P\right)=M^{*}\left(P\right)$ for every
$P\le S$.
\item $M_{*}\left(\overline{\varphi}\right)=M^{*}\left(\overline{\varphi}^{-1}\right)$
for every isomorphism $\overline{\varphi}$ in $\OF$.
\item For every $A,B\le C\le S$ then
\[
M^{*}\left(\overline{\iota_{B}^{C}}\right)M_{*}\left(\overline{\iota_{A}^{C}}\right)=\sum_{x\in\left[B\backslash C/A\right]}M_{*}\left(\overline{\iota_{B\cap\lui{x}{A}}^{B}c_{x}}\right)M^{*}\left(\overline{\iota_{B^{x}\cap A}^{A}}\right).
\]
\end{enumerate}
\end{defn}

It is often useful to view functors as modules over certain rings.
The following well known results allows us to do exactly that.
\begin{prop}[{\cite[Proposition 2.1]{WebbRepresentationAndCohomologyOfCategories}}]
\label{prop:equivalence-of-categories.}Let $\boldsymbol{C}$ be
a small category with finitely many objects and let $\R$ be a commutative
ring. The category of contravariant functors from $\boldsymbol{C}$
to $\R\Mod$ (resp. $\R\mod{}$) is isomorphic to $\R\boldsymbol{C}\Mod$
(resp. $\R\boldsymbol{C}\mod{}$) where $\R\boldsymbol{C}$ denotes
the category algebra of $\boldsymbol{C}$ over $\R$.

In view of Proposition \ref{prop:equivalence-of-categories.} we often
treat functors as modules and viceversa. More precisely we adopt the
following Notation.
\end{prop}

\begin{notation}
Let $\R$ be a commutative ring and let $\boldsymbol{C}$ be a small
category with finitely many objects. Abusing notation for every $M\in\R\boldsymbol{C}\Mod$
(resp. $\R\boldsymbol{C}\mod{}$) we denote also by $M$ the contravariant
functor from $\boldsymbol{C}$ to $\R\Mod$ (resp. $\R\mod{}$) corresponding
to $M$ via Proposition \ref{prop:equivalence-of-categories.}. Analogously,
for every contravariant functor $F:\boldsymbol{C}^{\text{op}}\to\R\Mod$
(resp. $\R\mod{}$) we denote also by $F$ the corresponding $\R\boldsymbol{C}$-module.
\end{notation}

Proposition \ref{prop:equivalence-of-categories.} also allows us
to translate the standard induction and restriction operations of
ring theory to functors. More precisely we introduce the following.
\begin{defn}
Let $\boldsymbol{C}\subseteq\boldsymbol{D}$ be small categories with
finitely many objects, let $M\in\R\boldsymbol{C}\Mod$ (resp. $\R\boldsymbol{C}\mod{}$)
and let $N\in\R\boldsymbol{D}\Mod$ (resp. $\R\boldsymbol{D}\mod{}$).
Using the natural inclusion $\R\boldsymbol{C}\subseteq\R\boldsymbol{D}$
we denote by $M\uparrow_{\boldsymbol{C}}^{\boldsymbol{D}}$ the $\R\boldsymbol{D}$-module
$M\uparrow_{\boldsymbol{C}}^{\boldsymbol{D}}:=M\otimes_{\R\boldsymbol{C}}\R\boldsymbol{D}$
and by $N\downarrow_{\boldsymbol{C}}^{\boldsymbol{D}}$ the $\R\boldsymbol{C}$-module
$N\downarrow_{\boldsymbol{C}}^{\boldsymbol{D}}:=N\otimes_{\R\boldsymbol{D}}\R\boldsymbol{C}$.
\end{defn}

This duality of functors and modules allows us to compute higher limits
of certain functors as $\Ext$ groups.
\begin{defn}
\label{def:Constant-functor.}Let $\boldsymbol{C}\subseteq\boldsymbol{D}$
be small categories with finitely many objects and let $\R$ be a
commutative ring. We denote by $\underline{\R}^{\boldsymbol{C}}\in\R\boldsymbol{C}\mod{}$
the module corresponding via Proposition \ref{prop:equivalence-of-categories.}
to the constant contravariant functor sending every object in $\boldsymbol{C}$
to the trivial $\R$-module $\R$ and every morphism to the identity.
We denote by $\underline{\R}_{\boldsymbol{C}}^{\boldsymbol{D}}:=\underline{\R}^{\boldsymbol{C}}\uparrow_{\boldsymbol{C}}^{\boldsymbol{D}}$
the $\R\boldsymbol{D}$-module induced from $\underline{\R}^{\boldsymbol{C}}$.
\end{defn}

\begin{lem}
\label{lem:lims-and-ext-groups.}Let $\R$ be a commutative ring,
let $S$ be a finite $p$-group, let $S'\le S$, let $\F'\subseteq\F$
be non necessarily saturated fusion systems over $S'$ and $S$ respectively,
let $\mathcal{C}$ be a family of subgroups of $S$ closed under $\F$-overconjugation
(i.e. such that $P\in\mathcal{C}$ and $\Hom_{\F}\left(P,Q\right)\not=\emptyset$
imply $Q\in\mathcal{C}$) and define $\mathcal{C}':=\left\{ P\in\mathcal{C}\,:\,P\le S'\right\} $.
For every contravariant functor $M:\OFC\to\R\mod{}$ and every $n\ge0$
the following natural isomorphisms hold
\[
\limn_{\OFD{\F'}{\mathcal{C}'}}M\downarrow_{\OFD{\F'}{\mathcal{C}'}}^{\OFC}\cong\Ext_{\R\OFD{\F'}{\mathcal{C}'}}^{n}\left(\underline{\R}^{\R\OFD{\F'}{\mathcal{C}'}},M\downarrow_{\OFD{\F'}{\mathcal{C}'}}^{\OFC}\right)\cong\Ext_{\R\OFD{\F'}{\mathcal{C}'}}^{n}\left(\underline{\R}_{\R\OFD{\F'}{\mathcal{C}'}}^{\R\OFC},M\right).
\]
\end{lem}

\begin{proof}
This is an immediate consequence of \cite[Corollary 5.2]{WebbRepresentationAndCohomologyOfCategories}
and \cite[Propositions 3.7 and 4.5]{HigherLimitsOverTheFusionOrbitCategory}.
\end{proof}
We conclude this section with the following notation on stable elements.
\begin{defn}
\label{def:MF}Let $S$ be a $p$-group, let $\F$ be a fusion system
over $S$, let $\mathcal{C}$ be a collection of subgroups of $S$,
let $\R$ be a commutative ring and let $M\in\R\OFC\mod{}$ (resp.
$M\in\R\OFC\Mod$). We define $M^{\F}:=\lim_{\OFC}M$.
\end{defn}

\begin{lem}
\label{lem:MF-as-subgroup-of-MS.}Let $S$ be a finite $p$-group,
let $\F$ be a fusion system over $S$, let $\mathcal{C}$ be a collection
of subgroups of $S$, let $\R$ be a commutative ring and let $M\in\R\OF\mod{}$
(resp. $\R\OF\Mod$). If $S\in\mathcal{C}$ we can view $M^{\F}$
as an $\R$-submodule of $M\left(S\right)$ via the following isomorphism.
\begin{equation}
M^{\F}\cong\left\{ x\in M\left(S\right)\,:\,\forall P\le S\text{ and }\forall\overline{\varphi},\overline{\psi}\in\Hom_{\OF}\left(P,S\right)\text{ then }M\left(\overline{\varphi}\right)\left(x\right)=M\left(\overline{\psi}\right)\left(x\right)\right\} .\label{eq:description-limits.}
\end{equation}
In particular for every $N\in\R\OFC[\FS]\mod{}$ (resp. $\R\OFC[\FS]\Mod$)
we have that $N^{\FS}\cong N\left(S\right)$.
\end{lem}

\begin{proof}
For every $P\in\mathcal{C}$ we have that $P\le S$ and, therefore,
$\overline{\iota_{P}^{Q}}\in\Hom_{\OF}\left(P,S\right)$. In particular
$\Hom_{\OFC}\left(P,S\right)\not=\emptyset$. The first part of the
statement now follows from the universal property of limits. The
second part of the statement follows from the first after noticing
that $S$ is in fact an initial object on $\OF[\FS]^{\text{op}}$.
\end{proof}

\section{\protect\label{sec:Dwyer-spaces}Induced Dwyer spaces.}

Let $G$ be a finite group and let $\overline{\mathcal{C}}$ be a
family of subgroups of $G$ closed under $G$ conjugation. In \cite{DwyerHomologyDecomposition}
Dwyer defines the topological space $X_{\overline{\mathcal{C}}}^{\delta}$
which is now known as the Dwyer space for the centralizer decomposition
of $G$. In \cite{HigherLimitsOverTheFusionOrbitCategory} Yalçin
works with a generalization of this space in the case of $G$ being
any discrete group. Doing so he proves that $X_{\overline{\mathcal{C}}}^{\delta}$
is homeomorphic, as a $G$ topological space, to the realization of
the poset category whose objects are the groups in $\overline{\mathcal{C}}$
and whose order is given by the natural inclusion (see \cite[Lemma 6.10]{HigherLimitsOverTheFusionOrbitCategory}).
Using this and the fact that the spaces $C_{G}\left(P\right)\backslash\left(X_{\overline{\mathcal{C}}}^{\delta}\right)^{P}$
are contractible for every $P\in\overline{\mathcal{C}}$ (see \cite[Proposition 6.12]{HigherLimitsOverTheFusionOrbitCategory})
he relates the homology groups of the spaces $C_{G}\left(P\right)\backslash\left(X_{\overline{\mathcal{C}}}^{\delta}\right)^{P}$
with higher limits over the category $\overline{\F}_{\overline{\mathcal{C}}}\left(G\right)$
whose objects are the groups in $\overline{\mathcal{C}}$ and whose
morphism sets are given by $\Hom_{\overline{\F}_{\overline{\mathcal{C}}}\left(G\right)}\left(P,Q\right):=Q\backslash\Hom_{G}\left(P,Q\right)$.

Our goal in this document is similar in nature, however, instead of
studying the higher limits of $\overline{\F}_{\overline{\mathcal{C}}}\left(G\right)$
directly we want to do it via higher limits over categories of the
form $\overline{\F}_{\overline{\mathcal{C}}_{H}}\left(H\right)$ for
some $H\le G$ and $\overline{\mathcal{C}}_{H}:=\left\{ P\in\overline{\mathcal{C}}\,:\,P\le H\right\} $.

To this end, for every $H\le G$ and every family $\mathcal{C}$ of
subgroups of $G$ closed under $G$ conjugation, we define the $G$
simplicial set $\NCGH{G}{H}{\overline{C}}$ (see Definition \ref{def:NCGH}).
This simplicial set is similar in nature to the nerve of the poset
category whose objects are the groups in $\overline{\mathcal{C}}$
and whose order is given by natural inclusion. This allows us to obtain
for $\NCGH{G}{H}{\overline{C}}$ results similar to those Yalçin obtains
for $X_{\overline{\mathcal{C}}}^{\delta}$. More precisely we prove
that for every $P\in\overline{\mathcal{C}}$ the induced simplicial
set $C_{G}\left(P\right)\backslash\left(\NCGH{G}{H}{\overline{C}}\right)^{P}$
has the same homology groups as a discrete set of points (see Lemma
\ref{lem:fusion-weak-equivalent-to-discrete-points.}). Using this
we are able to relate the homology groups of the spaces $C_{G}\left(P\right)\backslash\left(\NCGH{G}{H}{\overline{C}}\right)^{P}$
with the $\R\OFC[\FSG]$-module $\underline{\R}_{\OFD{\FAB{H\cap S}{H}}{\overline{\mathcal{C}}_{H\cap S}}}^{\OFC[\FSG]}$
(see Definition \ref{def:Constant-functor.}) for some finite $p$-group
$S\le G$, some family $\overline{\mathcal{C}}$ of subgroups of $S$
closed under $\FSG$ conjugation and some commutative ring $\R$.
In Section \ref{sec:main-section} we use this relation and Lemma
\ref{lem:lims-and-ext-groups.} in order to study higher limits over
$\OFD{\FAB{H\cap S}{H}}{\overline{\mathcal{C}}_{H\cap S}}$ and relate
them with higher limits over $\OFC[\FSG]$. Is this relation that
leads to the proof of Theorem \hyperref[thm:A]{A}.

Let us start by defining the $G$ simplicial sets $\NCGH{G}{H}{\overline{C}}$.
\begin{defn}
\label{def:PCGH.}Let $G$ be a discrete group, let $H\le G$ and
let $\overline{\mathcal{C}}$ be a family of subgroups of $G$ closed
under $G$ conjugation. We denote by $\PCGH{G}{H}{\overline{C}}$
the poset category whose objects are tuples of the form $\left(Hx,P\right)$
with $P\in\overline{\mathcal{C}}$ and $x\in N_{G}\left(P,H\right)$
and whose order is defined by setting $\left(Hx,P\right)\preceq\left(Hy,Q\right)$
if and only if $Hx=Hy$ and $P\le Q$. This category admits a compatible
$G$ action defined by setting $g\cdot\left(Hx,P\right):=\left(Hxg^{-1},\lui{g}{P}\right)$
for every $g\in G$ and every $\left(Hx,P\right)\in\text{Ob}\left(\PCGH{G}{H}{\overline{C}}\right)$.
\end{defn}

\begin{rem}
If $H=G$ the poset category $\PCGH{G}{H}{\overline{C}}$ is isomorphic
to the poset category whose objects are the groups in $\overline{\mathcal{C}}$
whose order is given by natural inclusion.
\end{rem}

\begin{defn}
\label{def:NCGH}With notation as in Definition \ref{def:PCGH.} we
define the simplicial set $\NCGH{G}{H}{\overline{C}}$ as the nerve
of $\PCGH{G}{H}{\overline{C}}$. This simplicial set inherits from
$\PCGH{G}{H}{\overline{C}}$ a compatible $G$ action that makes it
into a $G$ simplicial set.
\end{defn}

Let $P\in\overline{\mathcal{C}}$, the fixed points set $\left(\NCGH{G}{H}{\overline{C}}\right)^{P}$
of the $G$ simplicial set $\NCGH{G}{H}{\overline{C}}$ under the
action of $P$ is quite close to being contractible. More precisely
it is weak equivalent to a discrete collection of points as proven
by the following lemmas.
\begin{lem}
\label{lem:first-weak-equivalence.}Let $G$ be a discrete group,
let $H\le G$, let $\overline{\mathcal{C}}$ be any family of subgroups
of $G$ closed under $G$ conjugation and let $P\in\overline{\mathcal{C}}$.
There is an isomorphism of $N_{G}\left(P\right)$ simplicial sets
between the simplicial subset $\left(\NCGH{G}{H}{\overline{C}}\right)^{P}\subseteq\NCGH{G}{H}{\overline{C}}$
of points fixed under the action of $P$ and the nerve of the full
subcategory $\mathcal{P}$ of $\PCGH{G}{H}{\overline{C}}$ having
as objects tuples of the form $\left(Hx,Q\right)$ with $P\le N_{G}\left(Q\right)$
and $x\in N_{G}\left(PQ,H\right)$.
\end{lem}

\begin{proof}
Let $\left(\PCGH{G}{H}{\overline{C}}\right)^{P}$ be the full subcategory
of $\PCGH{G}{H}{\overline{C}}$ whose objects are fixed under the
action of $P$. We know from \cite[Proposition 5.11]{Dwyer2001HomotopyTheorwticMethodsInGroupCohomology}
that the natural map $\varphi:\left(\NCGH{G}{H}{\overline{C}}\right)^{P}\to\mathcal{N}\left(\left(\PCGH{G}{H}{\overline{C}}\right)^{P}\right)$
is an isomorphism of simplicial sets. Here $\mathcal{N}\left(\boldsymbol{C}\right)$
denotes the nerve of a category $\boldsymbol{C}$. Since the action
of $G$ on $\left(\NCGH{G}{H}{\overline{C}}\right)^{P}$ is inherited
from the action of $G$ on $\PCGH{G}{H}{\overline{C}}$ the isomorphism
$\varphi$ is in fact an isomorphism of $N_{G}\left(P\right)$ simplicial
sets.

Since $\PCGH{G}{H}{\overline{C}}$ is a poset category then there
exists at most one morphism between two objects. Therefore $\left(\PCGH{G}{H}{\overline{C}}\right)^{P}$
is the full subcategory of $\PCGH{G}{H}{\overline{C}}$ whose objects
are the tuples $\left(Hx,Q\right)$ such that $\left(Hxy^{-1},\lui{y}{Q}\right)=\left(Hx,Q\right)$
for every $y\in P$.

The identity $Q=\lui{y}{Q}$ is satisfied for every $y\in P$ if and
only if $P\le N_{G}\left(Q\right)$.

The identity $Hxy^{-1}=Hx$ is satisfied for every $y\in P$ if and
only if $xPx^{-1}\in H$ or, equivalently, if and only if $x\in N_{G}\left(P,H\right)$.

Since $x\in N_{G}\left(Q,H\right)$ by definition of $\PCGH{G}{H}{\overline{C}}$
and $N_{G}\left(Q,H\right)\cap N_{G}\left(P,H\right)=N_{G}\left(PQ,H\right)$
we conclude that $x\in N_{G}\left(PQ,H\right)$. The result follows.
\end{proof}
We can use Lemma \ref{lem:first-weak-equivalence.} in order to further
simplify the description of $\left(\NCGH{G}{H}{\overline{C}}\right)^{P}$
as follows.
\begin{lem}
\label{lem:second-weak-equivalence}Let $G$ be a discrete group,
let $H\le G$, let $\overline{\mathcal{C}}$ be any family of subgroups
of $G$ closed under $G$ conjugation and taking products (i.e. for
every $P,Q\in\overline{\mathcal{C}}$ if $PQ$ is a subgroup of $G$
then $PQ\in\overline{\mathcal{C}}$) and let $P\in\overline{\mathcal{C}}$.
There is a weak equivalence of $N_{G}\left(P\right)$ simplicial sets
between the simplicial subset $\left(\NCGH{G}{H}{\overline{C}}\right)^{P}\subseteq\NCGH{G}{H}{\overline{C}}$
of points fixed under the action of $P$ and the nerve of the full
subcategory $\PCGHP{G}{H}{\overline{C}}{P}$ of $\PCGH{G}{H}{\overline{C}}$
whose objects tuples of the form $\left(Hx,Q\right)$ with $P\le Q$.
\end{lem}

\begin{proof}
Let $\mathcal{P}$ be as in Lemma \ref{lem:first-weak-equivalence.},
let $\mathcal{P}':=\PCGHP{G}{H}{\overline{C}}{P}$, let $F:\mathcal{P}'\hookrightarrow\mathcal{P}$
be the natural inclusion of categories and define $F':\mathcal{P}\to\mathcal{P}'$
as the functor sending every $\left(Hx,Q\right)\in\mathcal{P}$ to
$F'\left(\left(Hx,Q\right)\right)=\left(Hx,PQ\right)$.

Since $PQ=Q$ for every $P\le Q$ we have that $F'\circ F=\Id_{\mathcal{P}'}$.
On the other hand, for every $\left(Hx,Q\right)\in\mathcal{P}$ we
have that $x\in N_{G}\left(PQ,H\right)$ and, therefore, there exists
a unique morphism from $\left(Hx,Q\right)$ to $\left(Hx,PQ\right)=F\left(F'\left(\left(Hx,Q\right)\right)\right)$
in $\mathcal{P}$. We deduce that there exists a natural transformation
$\eta:\Id_{\mathcal{P}}\Rightarrow F\circ F'$ defined by setting
$\eta_{\left(Hx,Q\right)}$ to be such morphism for every $\left(Hx,Q\right)\in\mathcal{P}$.
We conclude from \cite[Proposition 5.2]{Dwyer2001HomotopyTheorwticMethodsInGroupCohomology}
that the functors $F$ and $F'$ induce a weak equivalence between
$\mathcal{N}\left(P'\right)$ and $\mathcal{N}\left(P\right)$. The
result follows from Lemma \ref{lem:first-weak-equivalence.}.
\end{proof}
We can now prove that $\left(\NCGH{G}{H}{\overline{C}}\right)^{P}$
has the same homology groups as a discrete collection of points.
\begin{lem}
\label{lem:weak-equiv-to-discrete-points.}Let $G$ be a discrete
group, let $H\le G$, let $\overline{\mathcal{C}}$ be any family
of subgroups of $G$ closed under $G$ conjugation and taking products
and let $P\in\overline{\mathcal{C}}$. There is a weak equivalence
of $N_{G}\left(P\right)$ simplicial sets between the simplicial subset
$\left(\NCGH{G}{H}{\overline{C}}\right)^{P}\subseteq\NCGH{G}{H}{\overline{C}}$
of points fixed under the action of $P$ and the nerve of the full
subcategory $\PCGHP{G}{H}{\overline{C}}{P}$ of $\PCGH{G}{H}{\overline{C}}$
having as objects tuples of the form $\left(Hx,P\right)$ for some
$x\in N_{G}\left(P,H\right)$. That is $\left(\NCGH{G}{H}{\overline{C}}\right)^{P}$
is weak equivalent as an $N_{G}\left(P\right)$ simplicial set to
the nerve of the category whose are the cosets of $H\backslash N_{G}\left(P,H\right)$
and whose only morphisms are the identity morphisms.
\end{lem}

\begin{proof}
The proof is similar to that of Lemma \ref{lem:second-weak-equivalence}.
Let $\mathcal{P}':=\PCGH{G}{H}{\overline{C}}$, let $\mathcal{P}:=\PCGHP{G}{H}{\overline{C}}{P}$,
let $F:\mathcal{P}\hookrightarrow\mathcal{P}'$ be the natural inclusion
of categories and define $F':\mathcal{P}'\to\mathcal{P}$ as the functor
sending every $\left(Hx,Q\right)\in\mathcal{P}'$ to $F'\left(\left(Hx,Q\right)\right)=\left(Hx,P\right)$.

It is immediate from definition that $F'\circ F=\Id_{\mathcal{P}}$.
On the other hand, for every $\left(Hx,Q\right)\in\mathcal{P}$ we
have that $P\le Q$ and, therefore, $x\in N_{G}\left(Q,H\right)\subseteq N_{G}\left(P,H\right)$.
We deduce that for every $\left(Hx,Q\right)\in\mathcal{P}$ there
exists a unique morphism from $F\left(F'\left(\left(Hx,Q\right)\right)\right)=\left(Hx,P\right)$
to $\left(Hx,Q\right)$. We therefore obtain a natural transformation
$\eta:F\circ F'\Rightarrow\Id_{\mathcal{P}'}$ defined by taking $\eta_{\left(Hx,Q\right)}$
to be such morphism for every $\left(Hx,Q\right)\in\mathcal{P}'$.
We conclude from \cite[Proposition 5.2]{Dwyer2001HomotopyTheorwticMethodsInGroupCohomology}
that the functors $F$ and $F'$ induce a weak equivalence between
$\mathcal{N}\left(P'\right)$ and $\mathcal{N}\left(P\right)$. The
result now follows from Lemma \ref{lem:second-weak-equivalence}.
\end{proof}
Since we aim to study fusion systems induced by a group $G$ we are
mainly interested in the $N_{G}\left(P\right)$ simplicial sets on
which $C_{G}\left(P\right)$ and $P$ act trivially since, in these
cases, they can be viewed as $\Out_{G}\left(P\right)$ simplicial
sets. More precisely we are interested in the $\Out_{G}\left(P\right)$
simplicial set $C_{G}\left(P\right)\backslash\left(\NCGH{G}{H}{\overline{C}}\right)^{P}$.
\begin{lem}
\label{lem:fusion-weak-equivalent-to-discrete-points.}Let $G$ be
a discrete group, let $H\le G$, let $\overline{\mathcal{C}}$ be
any family of subgroups of $G$ closed under $G$ conjugation and
taking products and let $P\in\overline{\mathcal{C}}$. The $\Out_{G}\left(P\right)$
simplicial set $C_{G}\left(P\right)\backslash\left(\NCGH{G}{H}{\overline{C}}\right)^{P}$
is weak equivalent to the nerve of the category whose objects are
the equivalence classes of the quotient set $H\backslash\Hom_{G}\left(P,H\right)$
and whose only morphisms are the identity morphisms.
\end{lem}

\begin{proof}
From Lemma we know that $\left(\NCGH{G}{H}{\overline{C}}\right)^{P}$
is weak equivalent to the nerve of the $N_{G}\left(P\right)$ category
whose objects are indexed by the cosets $H\backslash N_{G}\left(P,H\right)$
and whose only morphisms are the identity morphisms. The result follows
after taking the quotient of this category by the action of $C_{G}\left(P\right)$
and noticing the isomorphism of $\Out_{G}\left(P\right)$ sets $H\backslash\Hom_{G}\left(P,H\right)\cong H\backslash N_{G}\left(P,H\right)/C_{G}\left(P\right)$.
\end{proof}
We can now use the $G$ action on $\NCGH{G}{H}{\overline{C}}$ in
order to define the following
\begin{defn}
\label{def:NNCGH}Let $S$ be a finite $p$-group, let $\F$ be a
(non necessarily saturated) fusion system over $S$, let $G$ be a
group realizing $\F$ (see \cite{LearyStancuRealisingFusionSystems,RobinsonAmalgams}),
let $H\le G$, let $\overline{\mathcal{C}}$ be any family of subgroups
of $G$ closed under $G$ conjugation and let $\mathcal{C}:=\left\{ P\in\overline{\mathcal{C}}\,:\,P\le S\right\} $.
We define $\NNCGH{G}{H}{C}:\OFC^{\text{op}}\to\SSets$ as the functor
sending every $P\in\mathcal{C}$ to the $\Out_{\F}\left(P\right)$
simplicial set $C_{G}\left(P\right)\backslash\left(\NCGH{G}{H}{\overline{C}}\right)^{P}$
and sending every morphism $c_{x}$ to the morphism naturally arising
from the action of $x$ on $\NCGH{G}{H}{\overline{C}}$.
\end{defn}

\begin{defn}
\label{def:Homology-funct-to-top.}Let $\R$ be a commutative ring,
let $n\ge0$ be an integer and let $G,H,\F$ and $\mathcal{C}$ be
as in Definition \ref{def:NNCGH}. We define the functor $\HGnHR{n}{\R}{H}:\OFC^{\text{op}}\to\R\text{-Mod}$
as the composition $\HGnHR{n}{\R}{H}:=H_{n}\left(-;\R\right)\circ\NNCGH{G}{H}{C}$.
\end{defn}

Using the descriptions of the $\Out_{G}\left(P\right)$ simplicial
sets $C_{G}\left(P\right)\backslash\left(\NCGH{G}{H}{C}\right)^{P}$
we can now provide the following description of the functors $\HGnHR{n}{\R}{H}$.
\begin{prop}
\label{prop:Hn-and-const-functor.}Let $\R$ be a commutative ring,
let $n\ge0$ be an integer, let $S$ be a finite $p$-group, let $\F$
be a (non necessarily saturated) fusion system over $S$, let $G$
be a group realizing $\F$, let $H\le G$, let $\overline{\mathcal{C}}$
be any family of subgroups of $G$ closed under $G$ conjugation and
taking products, let $\mathcal{C}:=\left\{ P\in\overline{\mathcal{C}}\,:\,P\le S\right\} $
and let $\mathcal{C}_{H}:=\left\{ P\in\overline{\mathcal{C}}\,:\,P\le S\cap H\right\} $.
Assume that for every $P\le H$ in $\overline{\mathcal{C}}$ there
exists $x\in H$ such that $\lui{x}{P}\in\mathcal{C}_{H}$. Then the
functor $\HGnHR{n}{\R}{H}$ equals $0$ for $n\ge1$ and is naturally
isomorphic to the functor $\underline{\R}_{\OFD{\FAB{H\cap S}{H}}{\mathcal{C}_{H}}}^{\OFC}:\OFC^{\text{op}}\to\R\text{-Mod}$
(see Definition \ref{def:Constant-functor.}) for $n=0$.
\end{prop}

\begin{proof}
From Lemma \ref{lem:fusion-weak-equivalent-to-discrete-points.} we
know that
\[
\HGnHR{n}{\R}{H}\left(P\right)=\begin{cases}
\R^{\left|H\backslash\Hom_{G}\left(P,H\right)\right|} & \text{if }n=0\\
0 & \text{else}
\end{cases}
\]
for every $P\in\mathcal{C}$. We can now define $\Hom_{\F}\left(P,\FAB{H\cap S}{H}\right)$
as the quotient of the Hom set $\Hom_{\F}\left(P,H\cap S\right)$
modulo the equivalence relation $\varphi\sim\psi$ if and only if
there exists $\theta\in\Hom_{\FAB{H\cap S}{H}}\left(\varphi\left(P\right),\psi\left(P\right)\right)$
such that $\theta\left(\varphi\left(x\right)\right)=\psi\left(x\right)$
for every $x\in P$. With this setup we have that $\underline{\R}_{\OFD{\FAB{H\cap S}{H}}{\mathcal{C}_{H}}}^{\OFC}\left(P\right)=\R^{\left|\Hom_{\F}\left(P,\FAB{H\cap S}{H}\right)\right|}$.

For every $P\in\mathcal{C}$ let $\Gamma_{P}:\Hom_{\F}\left(P,\FAB{H\cap S}{H}\right)\to H\backslash N_{G}\left(P,H\right)/C_{G}\left(P\right)$
be the morphism  defined by setting $\Gamma_{P}\left(\overline{c_{x}}\right)=HxC_{G}\left(P\right)$
for every $x\in N_{G}\left(P,H\cap S\right)$. It is immediate from
construction of $\Hom_{\F}\left(P,\FAB{H\cap S}{H}\right)$ that $\Gamma_{P}$
is injective. Since, by assumption, for every $x\in N_{G}\left(P,H\right)$
there exists $h\in H$ such that $hx\in N_{G}\left(P,H\cap S\right)$
we can conclude that $\Gamma_{P}$ is also surjective and, therefore,
an isomorphism. Thus $\Gamma_{P}$ leads to an isomorphism $\hat{\Gamma}_{P}:\underline{\R}_{\OFD{\FAB{H\cap S}{H}}{\mathcal{C}_{H}}}^{\OFC}\left(P\right)\bjarrow\HGnHR{0}{\R}{H}\left(P\right)$
which in turn lifts naturally to an isomorphism $\hat{\Gamma}:\underline{\R}_{\OFD{\FAB{H\cap S}{H}}{\mathcal{C}_{H}}}^{\OFC}\bjarrow\HGnHR{0}{\R}{H}$
thus concluding the proof.
\end{proof}
We conclude this section with the following result which allows us
to relate homology groups of the introduced simplicial sets with higher
limits over the fusion orbit category $\OFD{\FAB{H\cap S}{H}}{\mathcal{C}_{H}}$.
\begin{cor}
\label{cor:vanishing-ext}With notation as in Proposition \ref{prop:Hn-and-const-functor.},
for every $\R\OFC$-module $M$ and every $n\ge0$ we have that
\[
\Ext_{\R\OFC}^{n}\left(\HGnHR{t}{\R}{H},M\right)=\begin{cases}
\limn_{\OFD{\FAB{H\cap S}{H}}{\mathcal{C}_{H}}}\left(M\downarrow_{\OFD{\FAB{H\cap S}{H}}{\mathcal{C}}}^{\OFC}\right) & \text{if }t=0\\
0 & \text{else}
\end{cases}
\]
\end{cor}

\begin{proof}
From Proposition \ref{prop:Hn-and-const-functor.} we know that
\[
\Ext_{\R\OFC}^{n}\left(\HGnHR{t}{\R}{H},M\right)=\begin{cases}
\Ext_{\R\OFC}^{n}\left(\underline{\R}_{\OFD{\FAB{H\cap S}{H}}{\mathcal{C}_{H}}}^{\OFC},M\right) & \text{if }t=0\\
0 & \text{else}
\end{cases}
\]
the result now follows from Lemma \ref{lem:lims-and-ext-groups.}.
\end{proof}

\section{\protect\label{sec:main-section}An exact sequence for higher limits
over the fusion orbit category.}

This section is dedicated to proving Theorem \hyperref[thm:A]{A}.
In order to make the text easier to follow it might be useful to start
this section with an outline of the proof.

With notation as in Theorem \hyperref[thm:A]{A} let $\overline{H}:=\left\{ G_{1},G_{2}\right\} $
and let $\Dt$ be the poset category whose objects are the subgroups
$G_{1},G_{2}$ and $S$ of $G$ and whose order is given by natural
inclusion (see Definition \ref{def:Dn}) and let $\NCGH{G}{\overline{H}}{C}:\Dt^{\text{op}}\to\Top$
be a functor sending each $H\in\Dt$ to the geometric realization
of the simplicial set $\NCGH{G}{H}{C}$ of Definition \ref{def:NCGH}.
We know from \cite[Proposition 5.7]{HigherLimitsOverTheFusionOrbitCategory}
that there exists a first quadrant cohomology spectral sequence starting
at page $2$ of the form
\begin{equation}
\Ext_{\Fp[2]\OFc}^{s}\left(H_{t}\left(C_{G}\left(?\right)\backslash\left(\hocolim_{\Dt}\left(\NCGH{G}{\overline{H}}{\overline{C}}\right)\right)^{?};\R\right),M\right)\Rightarrow H_{\OFc}^{s+t}\left(\hocolim_{\Dt}\left(\NCGH{G}{\overline{H}}{\overline{C}}\right);M\right).\label{eq:ss-general.}
\end{equation}
where $M$ is an $\R\OFc$-module and $H_{\OFc}^{s+t}\left(X;M\right)$
denotes the fusion Bredon cohomology of $X$ with coefficients in
$M$ (see Definition \ref{def:fusion-Bredon-cohomology}).

Applying a result of Bousfield and Kan (see \cite[\S XII 4.5]{BousfieldKanHomotopyLimitsColimitsCompletionsAndLocalizations})
we obtain an additional spectral sequence converging to $H_{\OFc}^{s+t}\left(\hocolim_{\mathcal{P}}\left(\left|\NCGH{G}{\overline{H}}{\overline{C}}\right|\right);M\right)$.
If $M$ is as in the statement of Theorem \hyperref[thm:A]{A} we
can use this spectral in order to obtain the following isomorphisms
for $n\ge1$ (see Corollaries \ref{cor:Lim-D2.} and \ref{cor:iso-bredon-cohomology-hocolim})
\begin{equation}
H_{\OFc}^{n}\left(\hocolim_{\mathcal{P}}\left(\NCGH{G}{\overline{H}}{\overline{C}}\right);M\right)\cong\begin{cases}
M\left(S\right)/\left(M^{\F_{1}}+M^{\F_{2}}\right) & \text{if }n=1\\
0 & \text{if }n>1
\end{cases}\label{eq:simplification-bredon}
\end{equation}
On the other hand, using the results of Section \ref{sec:Dwyer-spaces}
and Lemma \ref{lem:lims-and-ext-groups.}, we prove in Corollary \ref{cor:homology-TTCGH}
that
\begin{equation}
\Ext_{\R\OFc}^{s}\left(H_{t}\left(C_{G}\left(?\right)\backslash\left(\hocolim_{\mathcal{P}}\left(\NCGH{G}{\overline{H}}{\overline{C}}\right)\right)^{?};\R\right),M\right)\cong\begin{cases}
\limn[s]_{\OFc}M & \text{if }t=0\\
\Ext_{\Fp\OFc}^{s}\left(\CGpC,M\right) & \text{if }t\ge1\\
0 & \text{else}
\end{cases}\label{eq:simplification-ext.}
\end{equation}
We can therefore deduce that all except the first two rows of the
second page of the spectral sequence of Equation (\ref{eq:ss-general.})
vanish. Theorem \hyperref[thm:A]{A} finally follows after using Equations
(\ref{eq:simplification-bredon}) and (\ref{eq:simplification-ext.})
to rewrite the spectral sequence of Equation (\ref{eq:ss-general.})
and applying \cite[Exercise 5.2.2]{WeibelIntroductionToHomologicalAlgebra1994}
to the resulting spectral sequence.

\medskip{}

In order to provide a precise version of the above outlined proof
it is useful to start by introducing the following notation.
\begin{defn}
\label{def:Dn}Let $n\ge1$. We define $\Dn$ as the poset category
whose objects are the nonempty subsets of $\left\{ 1,\dots,n\right\} $
and whose order is given by reverse inclusion. 
\end{defn}

We are mostly concerned with the study of higher limits of functors
over the category $\Dt$ which we use to obtain the isomorphisms of
Equation (\ref{eq:simplification-bredon}). The following Lemma is
very useful to compute them
\begin{lem}
\label{lem:proj-res-R.}Let $\R$ be a commutative ring, let $\underline{\R}^{\Dt}\in\R\Dt\mod{}$
be as in Definition \ref{def:Constant-functor.}, define the modules
$P_{0},P_{1}\in\R\Dt\mod{}$ and the morphisms $\pi:P_{0}\to\underline{\R}^{\Dt}$
and $\iota:P_{1}\to P_{0}$ by setting
\begin{align*}
P_{0}\left(\left\{ 1,2\right\} \right) & =\R^{2}, & P_{0}\left(\left\{ 1\right\} \right) & =P_{0}\left(\left\{ 2\right\} \right)=\R,\\
P_{1}\left(\left\{ 1,2\right\} \right) & =\R, & P_{1}\left(\left\{ 1\right\} \right) & =P_{1}\left(\left\{ 2\right\} \right)=0,
\end{align*}
letting $P_{0}\left(\left\{ 1,2\right\} \preceq\left\{ i\right\} \right):P_{0}\left(\left\{ i\right\} \right)\to P_{0}\left(\left\{ 1,2\right\} \right)$
be the natural inclusion in the $i^{\text{th}}$ component and setting
\begin{align*}
\pi_{\left\{ i\right\} }:\underset{x}{P_{0}\left(\left\{ i\right\} \right)} & \underset{\to}{\to}\underset{x}{\underline{\R}^{\Dt}\left(\left\{ i\right\} \right)}, & \pi_{\left\{ 1,2\right\} }:\underset{x\oplus y}{P_{0}\left(\left\{ 1,2\right\} \right)} & \underset{\to}{\to}\underset{x+y}{\underline{\R}^{\Dt}\left(\left\{ 1,2\right\} \right)},\\
\iota_{\left\{ 1,2\right\} }:\underset{x}{P_{1}\left(\left\{ 1,2\right\} \right)} & \underset{\to}{\to}\underset{x\oplus-x}{P_{0}\left(\left\{ 1,2\right\} \right)}.
\end{align*}

The following is a projective resolution of $\underline{\R}^{\Dt}$
\[
0\to P_{1}\stackrel{\iota}{\to}P_{0}\stackrel{\pi}{\to}\underline{\R}^{\Dt}\to0.
\]
\end{lem}

\begin{proof}
It is immediate from definition that $P_{0}\oplus P_{1}\cong\R\Dt$
as $\R\Dt\text{-modules}$. We can therefore conclude that both $P_{0}$
and $P_{1}$ are projective. On the other hand $\iota$ is clearly
injective and $\pi$ is clearly surjective with kernel $\ker\left(\pi\right)=\left\{ x\oplus-x\in P_{0}\left(\left\{ 1,2\right\} \right)\,:\,x\in\R\right\} =\im\left(\iota\right)$.
The result follows.
\end{proof}
Using Lemma \ref{lem:proj-res-R.} we can compute higher limits over
$\Dt$ as follows.
\begin{cor}
\label{cor:Lim-D2.}Let $\R$ be a commutative ring and let $M$ be
a $\R\Dt\text{-module}$. Then we have that
\[
\limn_{\Dt}M\cong\begin{cases}
\limn[]_{\Dt}M & \text{if }n=0\\
M\left(\left\{ 1,2\right\} \right)/\left(\im\left(M\left(\left\{ 1,2\right\} \preceq\left\{ 1\right\} \right)\right)+\im\left(M\left(\left\{ 1,2\right\} \preceq\left\{ 2\right\} \right)\right)\right) & \text{if }n=1\\
0 & \text{if }n\ge2
\end{cases}
\]
\end{cor}

\begin{proof}
It is well known that $\limn[0]_{\Dt}M\cong\lim_{\Dt}M$. This proves
the first identity.

We know from \cite[Corollary 5.2]{WebbRepresentationAndCohomologyOfCategories}
that for every $n\ge0$ then $\limn_{\Dt}M\cong\Ext_{\R\Dt}^{n}\left(\underline{\R}^{\Dt},M\right)$.
From Lemma \ref{lem:proj-res-R.} we also know that $\underline{\R}^{\Dt}$
has projective dimension $2$. Therefore, for every $n\ge2$, we have
that
\[
\limn_{\Dt}M\cong\Ext_{\R\Dt}^{n}\left(\underline{\R}^{\Dt},M\right)=0.
\]
Finally, with notation as in Lemma \ref{lem:proj-res-R.}, we have
that
\begin{align*}
\Hom_{\R\Dt}\left(P_{1},M\right) & \cong M\left(\left\{ 1,2\right\} \right), & \Hom_{\R\Dt}\left(P_{0},M\right) & \cong M\left(\left\{ 1\right\} \right)\oplus M\left(\left\{ 2\right\} \right).
\end{align*}
As a result we obtain the following isomorphisms which complete the
proof
\begin{align*}
\limn[1]_{\Dt}M & \cong\Ext_{\R\Dt}^{1}\left(\underline{\R}^{\Dt},M\right),\\
 & \cong\Hom_{\R\Dt}\left(P_{1},M\right)/\im\left(\iota^{*}:\Hom_{\R\Dt}\left(P_{0},M\right)\to\Hom_{\R\Dt}\left(P_{1},M\right)\right),\\
 & \cong M\left(\left\{ 1,2\right\} \right)/\left(\im\left(M\left(\left\{ 1,2\right\} \to\left\{ 1\right\} \right)\right)+\im\left(M\left(\left\{ 1,2\right\} \to\left\{ 2\right\} \right)\right)\right).
\end{align*}
\end{proof}
We can now define the following functor from $\Dn$ to the category
of simplicial sets.
\begin{defn}
\label{def:NCGH-functor}Let $G$ be a group, let $S\le G$ be a finite
$p$-group, let $\overline{\mathcal{C}}$ be a family of subgroups
of $G$ closed under $G$ conjugation, let $\mathcal{C}:=\left\{ P\in\overline{\mathcal{C}}\,:\,P\le S\right\} $,
let $n\ge1$ and let $\overline{H}:=\left\{ H_{i}\right\} _{i=1,\dots,n}$
be a finite family of subgroups of $G$. For every $P\in\mathcal{C}$
we define the functors $\NCGH{G}{\overline{H}}{\overline{C}}$ and
$\NNCGH{G}{\overline{H}}{\overline{C}}$ as the functors from $\Dn$
to the category of $G$ topological spaces by setting for every nonempty
subset $I\subseteq\left\{ 1,\dots,n\right\} $
\begin{align*}
\NCGH{G}{\overline{H}}{\overline{C}}\left(I\right) & :=\left|\NCGH{G}{H_{I}}{\overline{C}}\right|,\\
\NNCGH{G}{\overline{H}}{C}\left(P\right)\left(I\right) & :=\left|\NNCGH{G}{H_{I}}{C}\left(P\right)\right|,
\end{align*}
where $H_{I}:=\bigcap_{i\in I}H_{i}$ and $\NCGH{G}{H_{I}}{\overline{C}}$
and $\NNCGH{G}{H_{I}}{C}$ are as in Definitions \ref{def:NCGH} and
\ref{def:NNCGH}.
\end{defn}

We want to attribute to the homotopy colimit of $\NCGH{G}{\overline{H}}{\overline{C}}$
a $G$ action compatible with its simplicial set structure. In order
to do so we need to introduce the following.
\begin{defn}
\label{def:transport-category.}Let $\boldsymbol{C}$ be a small category
and let $F:\boldsymbol{C}\to\Top$ be a functor. The \textbf{transport
category of $F$} is the category $\Tr\left(F\right)$ whose objects
are tuples of the form $\left(c,x\right)$ with $c\in\boldsymbol{C}$
and $x\in F\left(c\right)$ and whose morphisms are given by
\[
\Hom_{\Tr\left(F\right)}\left(\left(c,x\right),\left(c',x'\right)\right):=\left\{ \varphi\in\Hom_{\boldsymbol{C}}\left(c,c'\right)\,:\,F\left(\varphi\right)\left(x\right)=x'\right\} .
\]
The composition of morphisms in $\Tr\left(F\right)$ is inherited
from the composition of morphisms in $\boldsymbol{C}$.
\end{defn}

Let $G,\overline{\mathcal{C}},\overline{H}$ and $n$ be as in Definition
\ref{def:NCGH-functor}. We know from \cite[Proposition 5.12]{Dwyer2001HomotopyTheorwticMethodsInGroupCohomology}
that the homotopy colimit $\hocolim_{\Dn}\left(\NCGH{G}{\overline{H}}{\overline{C}}\right)$
is naturally isomorphic to the geometric realization of the nerve
of $\Tr\left(\NCGH{G}{\overline{H}}{\overline{C}}\right)$. Thus,
using the $G$ set structure of $\NCGH{G}{\overline{H}}{\overline{C}}$,
for each $I\in\Dn$ we obtain the following.
\begin{lem}
\label{lem:hocolim-as-nerve-of-transport.}Let $G$ be a group, let
$\overline{\mathcal{C}}$ be a family of subgroups of $G$ closed
under $G$ conjugation, let $n\ge1$ and let $\overline{H}:=\left\{ H_{i}\right\} _{i=1,\dots,n}$
be a finite family of subgroups of $G$. The transport category $\Tr\left(\NCGH{G}{\overline{H}}{\overline{C}}\right)$
admits a $G$ action defined by setting $g\cdot\left(I,x\right)=\left(I,g\cdot x\right)$
for every $g\in G$ and every $\left(I,x\right)\in\Tr\left(\NCGH{G}{\overline{H}}{\overline{C}}\right)$.
This induces a $G$ action on the homotopy colimit $\hocolim_{\Dn}\left(\NCGH{G}{\overline{H}}{\overline{C}}\right)$.
\end{lem}

In order to study the homotopy colimit $\hocolim_{\Dn}\left(\NCGH{G}{\overline{H}}{\overline{C}}\right)$
it is useful to first introduce the following notation
\begin{defn}
\label{def:TCGH}Let $G$ be a group, let $S\le G$ be a finite $p$-group,
let $\F:=\FSG$ be a (non necessarily saturated) fusion system, let
$\overline{\mathcal{C}}$ be any family of subgroups of $G$ closed
under $G$ conjugation and taking products, let $n\ge1$, let $\overline{H}:=\left\{ H_{1},\dots,H_{n}\right\} $
be a family of subgroups of $G$ and for every $P\in\overline{\mathcal{C}}$
define $\boldsymbol{C}_{P}$ as the poset category whose objects are
tuples of the form $\left(I,H_{I}x\right)$ with $\emptyset\not=I\subseteq\left\{ 1,\dots,n\right\} $,
$H_{I}$ as in Definition \ref{def:NCGH-functor} and $x\in N_{G}\left(P,H_{I}\right)$
and whose order is given by setting $\left(I,H_{I}x\right)\preceq\left(J,H_{J}y\right)$
if and only if $J\subseteq I$ and $H_{J}x=H_{J}y$. We define:
\begin{itemize}
\item $\TCGH{G}{\overline{H}}{\overline{C}}:\OGC^{\text{op}}\to\Top$ as
the functor sending every $P\in\overline{\mathcal{C}}$ to the geometric
realization of the nerve of the category $\boldsymbol{C}_{P}$.
\item $\TTCGH{G}{\overline{H}}{C}:\OFC^{\text{op}}\to\Top$ as the functor
sending every $P\in\mathcal{C}$ to
\[
\TTCGH{G}{\overline{H}}{C}\left(P\right):=C_{G}\left(P\right)\backslash\TCGH{G}{\overline{H}}{\overline{C}}\left(P\right).
\]
\end{itemize}
\end{defn}

The following is one of the key result needed in order to compute
the homology groups of the homotopy colimit of $\NCGH{G}{\overline{H}}{\overline{C}}$.
\begin{lem}
\label{lem:weak-equiv-hocolim.}With notation as in Definition \ref{def:TCGH},
for every $P\in\mathcal{C}$ there exists a natural weak equivalence
of simplicial sets
\[
C_{G}\left(P\right)\backslash\left(\hocolim_{\Dn}\left(\NCGH{G}{\overline{H}}{\overline{C}}\right)\right)^{P}\simeq\TTCGH{G}{\overline{H}}{C}\left(P\right)
\]
\end{lem}

\begin{proof}
Throughout this proof we write $X:=C_{G}\left(P\right)\backslash\left(\hocolim_{\Dn}\left(\NCGH{G}{\overline{H}}{\overline{C}}\right)\right)^{P}$.
From Lemma \ref{lem:hocolim-as-nerve-of-transport.} we know that
$X$ is naturally isomorphic to $\left|C_{G}\left(P\right)\backslash\mathcal{N}\left(\Tr\left(\NCGH{G}{\overline{H}}{\overline{C}}\right)\right)^{P}\right|$.
Since the action of $G$ on $\left|\mathcal{N}\left(\Tr\left(\NCGH{G}{\overline{H}}{\overline{C}}\right)\right)\right|$
derives from the action of $G$ on $\Tr\left(\NCGH{G}{\overline{H}}{\overline{C}}\right)$
we can conclude that $X$ is naturally isomorphic to $\left|\mathcal{N}\left(C_{G}\left(P\right)\backslash\Tr\left(\NCGH{G}{\overline{H}}{\overline{C}}\right)^{P}\right)\right|$.
From the description of the action of $G$ on $\Tr\left(\NCGH{G}{\overline{H}}{\overline{C}}\right)^{P}$
(see Lemma \ref{lem:hocolim-as-nerve-of-transport.}) we can now deduce
that the category $C_{G}\left(P\right)\backslash\Tr\left(\NCGH{G}{\overline{H}}{\overline{C}}\right)^{P}$
is naturally isomorphic to the category $\Tr\left(\NNCGH{G}{\overline{H}}{C}\left(P\right)\right)$
(see Definition \ref{def:NNCGH}). We know from Lemma \ref{lem:fusion-weak-equivalent-to-discrete-points.}
that for every $I\in\Dn$ the simplicial set $\NNCGH{G}{\overline{H}}{C}\left(P\right)\left(I\right)$
is weak equivalent to the nerve of the category $\boldsymbol{C}_{P,I}$
whose objects are the double cosets of $H\backslash N_{G}\left(P,H_{I}\right)/C_{G}\left(P\right)$
and whose only morphisms are the identity morphisms. Therefore its
geometric realization is weak equivalent to the discrete topological
space with points indexed by $H\backslash N_{G}\left(P,H_{I}\right)/C_{G}\left(P\right)$.
Let $F:\Dn\to\Top$ be the functor that sends every $I\in\Dn$ to
this topological space. It follows from \cite[Remark 4.14 and Proposition 5.12]{Dwyer2001HomotopyTheorwticMethodsInGroupCohomology}
that $X$ is naturally weak equivalent to $\left|\mathcal{N}\left(\Tr\left(F\right)\right)\right|$.
Since $\left|\mathcal{N}\left(\Tr\left(F\right)\right)\right|$ is
precisely the topological space $\TTCGH{G}{\overline{H}}{C}\left(P\right)$
the result follows.
\end{proof}
Before proceeding let us introduce some useful notation.

Let $G_{1},G_{2}$ be finite groups, let $S$ be a $p$-group satisfying
$S\in\text{Syl}_{p}\left(G_{i}\right)$ for $i=1,2$, define $G:=G_{1}*_{S}G_{2}$,
let $P\le S$ and for every $x\in N_{G}\left(P,G_{i}\right)$ let
$\overline{G_{i}}x:=\left\{ G_{i}xz\,:\,z\in C_{G}\left(P\right)\right\} $.

Fix $i=1,2$. With notation as above we have that 
\[
G_{i}\backslash N_{G}\left(P,G_{i}\right)=\bigsqcup_{x\in\left[G_{i}\backslash N_{G}\left(P,G_{i}\right)/C_{G}\left(P\right)\right]}\overline{G_{i}}x,
\]
where $\left[G_{i}\backslash N_{G}\left(P,G_{i}\right)/C_{G}\left(P\right)\right]$
is a choice of representatives of the double cosets of $G_{i}\backslash N_{G}\left(P,G_{i}\right)/C_{G}\left(P\right)$.
Since $P$ is a $p$-group we can deduce from the second Sylow theorem
that for every $x\in N_{G}\left(P,G_{i}\right)$ there exists $y\in G_{i}$
such that $yx\in N_{G}\left(P,S\right)$. We can therefore choose
the representatives $x\in\left[G_{i}\backslash N_{G}\left(P,G_{i}\right)/C_{G}\left(P\right)\right]$
satisfying $x\in N_{G}\left(P,S\right)$. For every $x\in\left[G_{i}\backslash N_{G}\left(P,G_{i}\right)/C_{G}\left(P\right)\right]$
fix a subset $T_{x}\subseteq C_{G}\left(P\right)$ such that $\overline{G_{i}}x=\bigsqcup_{z\in T_{x}}G_{i}xz$.
With this setup we have that for every $y\in N_{G}\left(P,G_{i}\right)$
there exists a unique $x_{y}\in\left[G_{i}\backslash N_{G}\left(P,G_{i}\right)/C_{G}\left(P\right)\right]$
and a unique $z_{y}\in T_{x_{y}}$ such that $G_{i}y=G_{i}x_{y}z_{y}$.
Since $x_{y}z_{y}\in N_{G}\left(P,S\right)$ we can define the cosets
\begin{align*}
\left(G_{i}y\right)_{S} & :=Sx_{y}z_{y}\in S\backslash N_{G}\left(P,S\right), & \left(G_{i}yC_{G}\left(P\right)\right)_{S} & :=Sx_{y}C_{G}\left(P\right)\in S\backslash N_{G}\left(P,S\right)/C_{G}\left(P\right).
\end{align*}
Using this notation we can now prove the following.
\begin{lem}
\label{lem:Classifying=000020space.}Let $G_{1},G_{2}$ be finite
groups, let $S$ be a $p$-group satisfying $S\in\text{Syl}_{p}\left(G_{i}\right)$
for $i=1,2$, define $G:=G_{1}*_{S}G_{2}$, let $P\le S$ be $p$-centric
(i.e. $Z\left(P\right)$ is a maximal $p$-subgroup of $C_{G}\left(P\right)$),
for every $Sx\in S\backslash N_{G}\left(P,S\right)$ and every $SyC_{G}\left(P\right)\in S\backslash N_{G}\left(P,S\right)/C_{G}\left(P\right)$
let $e_{Sx}$ and $e_{SyC_{G}\left(P\right)}$ be the corresponding
standard basis element of $\mathbb{R}^{\left|S\backslash N_{G}\left(P,S\right)\right|}$
and $\mathbb{R}^{\left|S\backslash N_{G}\left(P,S\right)/C_{G}\left(P\right)\right|}$
respectively and, with notation as above, define the subspaces $E_{P,G}\subseteq\mathbb{R}^{\left|S\backslash N_{G}\left(P,S\right)\right|}$
and $B_{P,G}\subseteq\mathbb{R}^{\left|S\backslash N_{G}\left(P,S\right)/C_{G}\left(P\right)\right|}$
by setting
\begin{align*}
E_{P,G} & =\left\{ te_{Sx}+\left(1-t\right)e_{\left(G_{i}x\right)_{S}}\,:\,t\in\left[0,1\right],\,x\in N_{G}\left(P,S\right)\text{and }\exists i\in\left\{ 1,2\right\} \right\} ,\\
B_{P,G} & =\left\{ te_{SxC_{G}\left(P\right)}+\left(1-t\right)e_{\left(G_{i}xC_{G}\left(P\right)\right)_{S}}\,:\,t\in\left[0,1\right],\,x\in N_{G}\left(P,S\right)\text{and }\exists i\in\left\{ 1,2\right\} \right\} .
\end{align*}
Then:
\begin{enumerate}
\item \label{enu:cover.}$E_{P,G}$ is weakly contractible and admits a
free action of $C_{G}\left(P\right)/Z\left(P\right)$.
\item \label{enu:classifying-space.}$B_{P,G}$ is the classifying space
of $C_{G}\left(P\right)/Z\left(P\right)$.
\end{enumerate}
\end{lem}

\begin{proof}
Let $C_{G}\left(P\right)$ act on $E_{P,G}$ by setting $g\cdot e_{Sx}=e_{Sxg^{-1}}$
for every $x\in N_{G}\left(P,S\right)$. From construction of $\left(G_{i}xC_{G}\left(P\right)\right)_{S}$
and $\left(G_{i}x\right)_{S}$ we have that $\left(G_{i}xC_{G}\left(P\right)\right)_{S}=\left(G_{i}x\right)_{S}C_{G}\left(P\right)$.
We can therefore conclude that $B_{P,G}$ is the quotient of $E_{P,G}$
by the defined $C_{G}\left(P\right)$ action. In particular Item (\ref{enu:classifying-space.})
follows from Item (\ref{enu:cover.}).

\medskip{}

To prove Item (\ref{enu:cover.}) let us start by proving that the
above described action of $C_{G}\left(P\right)$ on $E_{P,G}$ induces
a free action of $C_{G}\left(P\right)/Z\left(P\right)$. From the
above description we have that the action of an element $y\in C_{G}\left(P\right)$
fixes an element $e_{Sx}\in E_{P,G}$ if and only if $xy^{-1}x^{-1}\in S$
that is if and only if $xy^{-1}x^{-1}\in C_{S}\left(\lui{x}{P}\right)$.
Since $\lui{x}{P}\le S$ is $p$-centric and $S$ is a $p$-group
we can deduce that $C_{S}\left(\lui{x}{P}\right)=Z\left(\lui{x}{P}\right)$.
In other words the action of $y$ fixes $e_{Sx}$ if and only if $y\in Z\left(P\right)$.
We can therefore conclude that the action of $C_{G}\left(P\right)$
on $E_{P,G}$ defines a free action of $C_{G}\left(P\right)/Z\left(P\right)$
on $E_{P,G}$.

\smallskip{}

We are now only left with proving that $E_{P,G}$ is weakly contractible.
Let's start by proving that $E_{P,G}$ is path connected. For every
$x\in N_{G}\left(P,S\right)$, every $i=1,2$ and every $y\in G_{i}$
such that $yx\in N_{G}\left(P,S\right)$ define the path $\pi_{Sx}^{Syx}:\left[0,1\right]\to E_{P,G}$
by setting 
\[
\pi_{Sx}^{Syx}\left(t\right)=\begin{cases}
\left(1-2t\right)e_{Sx}+2te_{\left(G_{i}x\right)_{3}} & \text{if }0\le t\le\frac{1}{2}\\
2\left(1-t\right)e_{\left(G_{i}yx\right)_{3}}+\left(2t-1\right)e_{Syx} & \text{if }\frac{1}{2}<t\le1
\end{cases}.
\]

Let $n\ge1$ and let $\left(a_{1},\dots,a_{n}\right)$ be a sequence
of elements satisfying $a_{i}\in G_{1}\cup G_{2}$ and $a_{i}\cdots a_{1}x\in N_{G}\left(P,S\right)$
for every $i=1,\dots,n$. We denote by $\pi_{Sx}^{\left(a_{1},\dots,a_{n}\right)}$
the path from $e_{Sx}$ to $e_{Sa_{i}\cdots a_{1}x}$ in $E_{P,G}$
obtained by ``composing'' the paths $\pi_{Sa_{i-1}\cdots a_{1}x}^{Sa_{i}\cdots a_{1}x}$
and by $\pi_{Sx}^{\left(\right)}$ the constant path to $e_{Sx}$.
It is immediate from construction of $E_{P,G}$ that every path in
$E_{P,G}$ from $e_{Sx}$ to $e_{Sa_{n}\cdots a_{1}x}$ is homotopy
equivalent to a path of the form $\pi_{Sx}^{\left(a_{1},\dots,a_{n}\right)}$.

Let $x\in N_{G}\left(P,S\right)$. From \cite[Theorem 1]{TreesSerre}
we know that there exists $n\ge0$, an element $a_{0}^{2}\in G_{2}$,
an element $a_{n+1}^{1}\in G_{1}$ and for every $i=1,2$ and every
$j=1,\dots,n$ an element $a_{j}^{i}\in G_{i}-S$ such that $x=a_{n+1}^{1}a_{n}^{2}a_{n}^{1}a_{n-1}^{2}\cdots a_{1}^{1}a_{0}^{2}$.
Since $S\le G_{1},G_{2}$ we can now deduce from \cite[Lemma 1]{RobinsonAmalgams}
that $a_{0}^{2}\in N_{G}\left(P,S\right)$, that $a_{j}^{2}a_{j}^{1}\cdots a_{1}^{1}a_{0}^{2}\in N_{G}\left(P,S\right)$
and that $a_{j}^{1}\cdots a_{1}^{1}a_{0}^{2}\in N_{G}\left(P,S\right)$
for every $j=1,\dots,n$. It follows that $\pi_{S}^{\left(a_{0}^{2},\dots,a_{n+1}^{1}\right)}$
is a path from $e_{S}$ to $e_{Sx}$ in $E_{P,G}$. Since this reasoning
is valid for every $x\in N_{G}\left(P,S\right)$ we can conclude that
$E_{P,G}$ is path connected.

Let's now prove that the fundamental group of $E_{P,G}$ is trivial.
From the above discussion we know that every loop in $E_{P,G}$ is
homotopy equivalent to a loop of the form $\pi_{S}^{\left(a_{1},\dots,a_{n}\right)}$
with $a_{n}\cdots a_{1}\in S$. If there exists $i\in\left\{ 1,\dots,n\right\} $
such that $a_{i}\in S$ then the path $\pi_{S}^{\left(a_{1},\dots,a_{n}\right)}$
if homotopy equivalent to the path $\pi_{S}^{\left(a_{1},\dots,a_{i-1},a_{i+1},\dots,a_{n}\right)}$.
Moreover, if there exist $i\in\left\{ 1,\dots,n-1\right\} $ and $j\in\left\{ 1,2\right\} $
such that $a_{i},a_{i+1}\in G_{j}$ then the path $\pi_{S}^{\left(a_{1},\dots,a_{n}\right)}$
if homotopy equivalent to the path $\pi_{S}^{\left(a_{1},\dots,a_{i}a_{i+1},\dots,a_{n}\right)}$.
Combining both these observations we can assume without loss of generality
that the $a_{i}$ are such that $a_{i}\in\left(G_{1}\cup G_{2}\right)-S$
and $a_{i}\in G_{j}$ for some $j\in\left\{ 1,2\right\} $ if and
only if $a_{i+1}\in G_{3-j}$.

For every $j=1,2$ fix representatives $\left[S\backslash G_{j}\right]$
of the cosets in $S\backslash G_{j}$. We can assume without loss
of generality that for every $i=1,\dots,n$ we have $a_{i}\in\left[S\backslash G_{j}\right]$
for some $j\in\left\{ 1,2\right\} $. Since $a_{n}\cdots a_{1}\in S$
we can conclude from \cite[Theorem 1]{TreesSerre} that $\pi_{S}^{\left(a_{1},\dots,a_{n}\right)}=\pi_{S}^{\left(\right)}$
is the constant path. Therefore the fundamental group of $E_{P,G}$
is trivial.

Finally it is clear from definition of $E_{P,G}$ that every higher
homotopy group of $E_{P,G}$ is trivial. We can therefore conclude
that $E_{P,G}$ is weakly contractible thus proving Item (\ref{enu:cover.})
and completing the proof.
\end{proof}
We can now finally provide a description of the homology groups of
$\TTCGH{G}{\overline{H}}{C}\left(P\right)$ and, therefore, of $C_{G}\left(P\right)\backslash\left(\hocolim_{\Dn}\left(\NCGH{G}{\overline{H}}{\overline{C}}\right)\right)^{P}$.
\begin{prop}
\label{prop:TTCGH-is-BGCP}Let $G_{1},G_{2}$ be finite groups, let
$S$ be a $p$-group satisfying $S\in\text{Syl}_{p}\left(G_{i}\right)$
for $i=1,2$, define $G:=G_{1}*_{S}G_{2}$, define $\F:=\FSG$, let
$\overline{\mathcal{C}}$ be the collection of all $p$-centric subgroups
of $G$, let $\mathcal{C}$ be the collection of all $\F$-centric
subgroups of $S$ and let $\overline{H}:=\left\{ G_{1},G_{2}\right\} $.
For every $P\in\mathcal{C}$ the topological space $\TTCGH{G}{\overline{H}}{C}\left(P\right)$
is weak equivalent to the classifying space of $C_{G}\left(P\right)/Z\left(P\right)$.
\end{prop}

\begin{proof}
Fix $P\in\mathcal{C}$, define $G_{3}:=S$, define $N:=\sum_{i=1}^{3}\left|G_{i}\backslash N_{G}\left(P,G_{i}\right)/C_{G}\left(P\right)\right|$
and view the topological space $B_{P,G}$ of Lemma \ref{lem:Classifying=000020space.}
as a subspace of $\mathbb{R}^{N}$. By ``blowing up'' on the points
$e_{\left(G_{i}xC_{G}\left(P\right)\right)_{3}}\in B_{P,G}$ we obtain
that $B_{P,G}$ is homotopy equivalent to the subspace $B'$ of $\mathbb{R}^{N}$
defined by setting
\[
B'=\left\{ te_{SxC_{G}\left(P\right)}+\left(1-t\right)e_{G_{i}xC_{G}\left(P\right)}\,:\,t\in\left[0,1\right],\,x\in N_{G}\left(P,S\right)\text{and }\exists i\in\left\{ 1,2\right\} \right\} .
\]
Let $\boldsymbol{C}$ be the category with objects tuples of the form
$\left(H,HxC_{G}\left(P\right)\right)$ for some $H\in\left\{ S,G_{1},G_{2}\right\} $
and $x\in N_{G}\left(P,S\right)$ and whose only morphisms are the
identities and a unique morphisms from $\left(S,SxC_{G}\left(P\right)\right)$
to $\left(G_{i},G_{i}yC_{G}\left(P\right)\right)$ if and only if
$SxC_{G}\left(P\right)\subseteq G_{i}yC_{G}\left(P\right)$. It is
immediate from definition that, given two composable morphisms in
$\boldsymbol{C}$, then at least one of them is the identity. We can
deduce that $B'$ is weak equivalent to the geometric realization
of $\boldsymbol{C}$.

On the other hand, since $S\in\Syl_{p}\left(G_{i}\right)$ for $i=1,2$,
then, for every double coset $X\in G_{i}\backslash N_{G}\left(P,G_{i}\right)/C_{G}\left(P\right)$,
there exists $x\in N_{G}\left(P,S\right)$ such that $X=G_{i}xC_{G}\left(P\right)$.
We can therefore conclude that $\TTCGH{G}{\overline{H}}{C}\left(P\right)$
is the geometric realization of $\boldsymbol{C}$ and, therefore,
weak equivalent to $B'$ which in turn is weak equivalent to $B_{P,G}$.
The result follows from Lemma \ref{lem:Classifying=000020space.}.
\end{proof}
\begin{cor}
\label{cor:homology-TTCGH}With notation as in Proposition \ref{prop:TTCGH-is-BGCP},
for every commutative ring $\R$ and every $P\in\mathcal{C}$ we have
that:

\[
H_{n}\left(\TTCGH{G}{\overline{H}}{C}\left(P\right),\R\right)=\begin{cases}
\R & \text{if }n=0\\
\text{Ab}\left(C_{G}\left(P\right)/Z\left(P\right)\right)\otimes\R & \text{if }n=1\\
0 & \text{else}
\end{cases}
\]
where $\text{Ab}\left(H\right)$ denotes the abelianization of $H$
for every group $H$. More precisely $H_{0}\left(\TTCGH{G}{\overline{H}}{C}\left(?\right),\R\right)=\underline{\R}$.
\end{cor}

\begin{proof}
From Lemma \ref{prop:TTCGH-is-BGCP} we know that 
\begin{align*}
H_{0}\left(\TTCGH{G}{\overline{H}}{C}\left(?\right),\R\right) & =\underline{\R},\\
H_{1}\left(\TTCGH{G}{\overline{H}}{C}\left(P\right),\Z\right) & =\text{Ab}\left(C_{G}\left(P\right)/Z\left(P\right)\right),\\
H_{n}\left(\TTCGH{G}{\overline{H}}{C}\left(P\right),\R\right) & =0\text{ for }n\ge2.
\end{align*}
Since $H_{0}\left(\TTCGH{G}{\overline{H}}{C}\left(P\right),\R\right)\cong\R$
is a free $\R$-module the result follows from the universal coefficients
theorem for homology (see \cite[Theorem 3A.3 and Proposition 3A.5]{HatcherAlgebraicTopology}).
\end{proof}
In light of the result shown in Corollary \ref{cor:homology-TTCGH}
it is now useful to introduce the following notation
\begin{defn}
\label{def:CGpC}With notation as in Corollary \ref{cor:homology-TTCGH}
we denote by $\CGpC:\OFC^{\text{op}}\to\R\text{-mod}$ the composition
$\OFC^{\text{op}}\stackrel{\TTCGH{G}{\overline{H}}{C}}{\longrightarrow}\text{Top}\stackrel{H_{1}\left(-;\R\right)}{\longrightarrow}\R\text{-mod}$.
\end{defn}

Corollary \ref{cor:homology-TTCGH} is very useful towards computing
the fusion Bredon cohomology of $\TTCGH{G}{\overline{H}}{C}$
\begin{defn}[{\cite[Definition 1.5]{HigherLimitsOverTheFusionOrbitCategory}}]
\label{def:fusion-Bredon-cohomology}Let $\R$ be a commutative ring,
let $S$ be a finite $p$-group, let $\mathcal{C}$ be a collection
of subgroups of $S$, let $\F$ be a fusion system over $S$, let
$G$ be a group realizing $S$ (i.e. $\F=\FSG$), let $X$ be a $G$
simplicial set, let $P_{*,*}$ be the Cartan-Eilenberg projective
resolution of the chain complex $C_{*}\left(C_{G}\left(?\right)\backslash\left|X\right|^{?};\R\right)$
and let $M$ be an $\R\OFC$-module. We define the \textbf{fusion
Bredon cohomology} of $\left|X\right|$ with coefficients in $M$
as
\[
H_{\OFC}^{*}\left(\left|X\right|;M\right):=H^{*}\left(\Hom_{\R\OFC}\left(\Tot^{\oplus}\left(P_{*,*}\right);M\right)\right).
\]
\end{defn}

In certain cases the fusion Bredon cohomology might be more easily
described in terms of limits.
\begin{lem}
\label{lem:iso-Bredon-cohomology-H}Let $\R$ be a commutative ring,
let $n\ge0$ be an integer, let $S$ be a finite $p$-group, let $\F$
be a (non necessarily saturated) fusion system over $S$, let $G$
be a group realizing $\F$, let $H\le G$, let $\overline{\mathcal{C}}$
be any family of subgroups of $G$ closed under $G$ conjugation and
taking products, let $\mathcal{C}:=\left\{ P\in\overline{\mathcal{C}}\,:\,P\le S\right\} $
and let $\mathcal{C}_{H}:=\left\{ P\in\overline{\mathcal{C}}\,:\,P\le S\cap H\right\} $.
Assume that for every $P\le H$ in $\overline{\mathcal{C}}$ there
exists $x\in H$ satisfying $\lui{x}{P}\in\mathcal{C}_{H}$. For every
$\R\OFC$-module $M$ the following isomorphism holds
\[
\limn_{\OFD{\FAB{H\cap S}{H}}{\mathcal{C}}}\left(M\downarrow_{\OFD{\FAB{H\cap S}{H}}{\mathcal{C}}}^{\OFC}\right)\cong H_{\OFC}^{s+t}\left(\NCGH{G}{H}{\overline{C}};M\right).
\]
\end{lem}

\begin{proof}
From \cite[Proposition 5.7]{HigherLimitsOverTheFusionOrbitCategory}
we know that there exists a spectral sequence of the form
\[
\Ext_{\R\OFC}^{s}\left(\HGnHR{t}{\R}{H},M\right)\Rightarrow H_{\OFC}^{s+t}\left(\NCGH{G}{H}{\overline{C}};M\right).
\]
The result follows from Corollary \ref{cor:vanishing-ext}.
\end{proof}
As a consequence of Lemma \ref{lem:iso-Bredon-cohomology-H} we obtain
the following spectral sequence.
\begin{lem}
\label{lem:ss-Bredon-cohomology-hocolim}Let $\R$ be a commutative
ring, let $S$ be a finite $p$-group, let $\F$ be a saturated fusion
system over $S$, let $G$ be a group realizing $\F$, let $\overline{\mathcal{C}}$
be any family of subgroups of $G$ closed under $G$ conjugation and
taking products, let $\mathcal{C}:=\left\{ P\in\overline{\mathcal{C}}\,:\,P\le S\right\} $,
let $\overline{H}:=\left\{ G_{1},G_{2}\right\} $ be a family of subgroups
of $G$ such that $S\le G_{i}$ for every $i=1,2$ and define the
(non necessarily saturated) fusion systems $\F_{\left\{ i\right\} }:=\FAB{S}{G_{i}}$
and $\F_{\left\{ 1,2\right\} }:=\FS$. Assume that for every $P\le G_{i}$
in $\overline{\mathcal{C}}$ there exists $x\in G_{i}$ such that
$\lui{x}{P}\le\mathcal{C}$. For every $\R\OFC$ module $M$ there
exists a spectral sequence of the form
\[
\limn[s]_{\Dt}\left(\limn[t]_{\OFC[\F_{?}]}\left(M\downarrow_{\OFC[\F_{?}]}^{\OFC}\right)\right)\Rightarrow H_{\OFC}^{n}\left(\hocolim_{\Dt}\left(\NCGH{G}{\overline{H}}{\overline{C}}\right);M\right).
\]
\end{lem}

\begin{proof}
From \cite[\S XII 4.5]{BousfieldKanHomotopyLimitsColimitsCompletionsAndLocalizations}
we know that there exists a spectral sequence of the form
\[
\limn[s]_{\Dt}\left(H_{\OFC}^{n}\left(\NCGH{G}{\overline{H}}{\overline{C}}\left(?\right);M\right)\right)\Rightarrow H_{\OFC}^{s+t}\left(\hocolim_{\Dt}\left(\NCGH{G}{\overline{H}}{\overline{C}}\right);M\right).
\]
The result now follows from Lemma \ref{lem:iso-Bredon-cohomology-H}.
\end{proof}
We can now use Corollary \ref{cor:homology-TTCGH} and Lemma \ref{lem:ss-Bredon-cohomology-hocolim}
in order to compute the fusion Bredon cohomology of $\hocolim_{\Dt}\left(\NCGH{G}{\overline{H}}{\overline{C}}\right)$.
\begin{cor}
\label{cor:iso-bredon-cohomology-hocolim}With notation as in Lemma
\ref{lem:ss-Bredon-cohomology-hocolim} if $\limn_{\OFC[\F_{I}]}\left(M\downarrow_{\OFC[\F_{I}]}^{\OFC}\right)=0$
for every $n\ge1$ and every $I\in\Dt$ then
\[
\limn_{\Dt}\left(\lim_{\OFC[\F_{?}]}\left(M\downarrow_{\OFC[\F_{?}]}^{\OFC}\right)\right)\cong H_{\OFC}^{n}\left(\hocolim_{\Dt}\left(\NCGH{G}{\overline{H}}{\overline{C}}\right);M\right).
\]
\end{cor}

\begin{proof}
Since $\limn_{\OFC[\F_{I}]}\left(M\downarrow_{\OFC[\F_{I}]}^{\OFC}\right)=0$
for every $n\ge1$ then the spectral sequence of Lemma \ref{lem:ss-Bredon-cohomology-hocolim}
is sharp. The result follows.
\end{proof}
The following result is stated as Theorem \hyperref[thm:A]{A} in
Section \ref{sec:Introduction.} and with it we conclude this section.
\begin{thm}
\label{thm:Exact-sequences-for-sharpness.}Let $S$ be a finite $p$-group,
let $G_{1},G_{2}$ be finite groups such that $S\in\Syl_{p}\left(G_{i}\right)$
for $i=1,2$, define $G:=G_{1}*_{S}G_{2}$, define $\F:=\FSG$, for
$i=1,2$ define $\F_{\left\{ i\right\} }:=\FSG[G_{i}]$, define $\F_{\left\{ 1,2\right\} }:=\FS$
and let $\mathcal{C}$ be the family of $\F$-centric subgroups of
$S$. For every $\Fp\OFc$-module $M$ such that $\limn_{\OFC[\F_{I}]}\left(M\downarrow_{\OFC[\F_{I}]}^{\OFc}\right)=0$
for every $n\ge1$ and every $I\in\Dt$ there is an isomorphism
\[
\Ext_{\Fp\OFc}^{n}\left(\CGpC,M\right)\cong\limn[n+2]_{\OFc}\left(M\downarrow_{\OFc}^{\OFc}\right),
\]
where $\CGpC$ is as in Definition \ref{def:CGpC}. Moreover there
is a short exact sequence of the form 
\[
0\to\underset{\OFc}{\limn[1]}\left(M\right)\to M\left(S\right)/\left(M^{\F_{\left\{ 1\right\} }}+M^{\F_{\left\{ 2\right\} }}\right)\to\underset{\Fp\OFc}{\Hom}\left(\CGpC,M\right)\to\underset{\OFc}{\limn[2]}\left(M\right)\to0.
\]
Where the $\Fp$-modules $M^{\F_{\left\{ i\right\} }}$ (see Definition
\ref{def:MF}) are seen as subgroups of $M\left(S\right)$ via Lemma
\ref{lem:MF-as-subgroup-of-MS.}.
\end{thm}

\begin{proof}
Let $\overline{H}:=\left\{ G_{1},G_{2}\right\} $. We know from \cite[Proposition 5.7]{HigherLimitsOverTheFusionOrbitCategory}
and Lemma \ref{lem:weak-equiv-hocolim.} that there exists a spectral
sequence of the form
\[
E_{2}^{s,t}:=\Ext_{\Fp\OFc}^{s}\left(H_{t}\left(\TTCGH{G}{\overline{H}}{C};\Fp\right),M\right)\Rightarrow H_{\OFc}^{s+t}\left(\underset{\Dt}{\hocolim}\left(\NCGH{G}{\overline{H}}{\overline{C}}\right);M\right).
\]
Applying Corollary \ref{cor:iso-bredon-cohomology-hocolim} we can
rewrite the above spectral sequence as
\[
E_{2}^{s,t}\Rightarrow\limn[s+t]_{\Dt}\left(\lim_{\OFC[\F_{?}]}\left(M\downarrow_{\OFC[\F_{?}]}^{\OFC}\right)\right).
\]
On the other hand, from Corollary \ref{cor:homology-TTCGH}, we know
that
\[
E_{2}^{s,t}=\begin{cases}
\Ext_{\Fp\OFc}^{s}\left(\CGpC,M\right) & \text{if }t=1\\
\Ext_{\Fp\OFc}^{s}\left(\underline{\Fp},M\right) & \text{if }t=0\\
0 & \text{else}
\end{cases}.
\]

Since $\Ext_{\Fp\OFc}^{n}\left(\underline{\Fp},M\right)\cong\limn_{\OFc}M$
(see Lemma \ref{lem:lims-and-ext-groups.}) we can apply \cite[Exercise 5.2.2]{WeibelIntroductionToHomologicalAlgebra1994}
in order to obtain the following long exact sequence
\[
\cdots\to\limn_{\Dt}\left(\lim_{\OFC[\F_{?}]}M\right)\to\Ext_{\R\OFc}^{n-1}\left(\CGpC,M\right)\to\limn[n+1]_{\OFc}\left(M\right)\to\limn[n+1]_{\Dt}\left(\lim_{\OFC[\F_{?}]}M\right)\to\cdots
\]
Where we define $\lim_{\OFC[\F_{?}]}M:=\lim_{\OFC[\F_{?}]}\left(M\downarrow_{\OFC[\F_{?}]}^{\OFC}\right)$
in order to simplify notation. From Corollary \ref{cor:Lim-D2.} we
know that $\limn_{\Dt}\left(\lim_{\OFC[\F_{?}]}\left(M_{?}\right)\right)=0$
for every $n\ge2$ and that $\limn[1]_{\Dt}\left(\lim_{\OFC[\F_{?}]}M\right)\cong M^{\F_{\left\{ 1,2\right\} }}/\left(M^{\F_{\left\{ 1\right\} }}+M^{\F_{\left\{ 2\right\} }}\right)$.
Since $M^{\F_{\left\{ 1,2\right\} }}=M\left(S\right)$ (see Lemma
\ref{lem:MF-as-subgroup-of-MS.}) the result follows.
\end{proof}

\section{\protect\label{sec:last-section}Applications of Theorem \hyperref[thm:A]{A}.}

This last section is dedicated to viewing potential ways in which
Theorem \hyperref[thm:A]{A} can be applied to study the sharpness
conjecture for the Benson-Solomon fusion systems.

Let us start by recalling a few facts regarding the definition of
this family of fusion systems. Let $q$ be an odd prime, let $n$
be a positive integer, let $S:=S\left(q^{n}\right)\in\Syl_{2}\left(\Spin_{7}\left(q^{n}\right)\right)$
and let $S_{0}:=S_{0}\left(q^{n}\right)\le S$ be the subgroup of
index $2$ described in \cite[Definition 1.7]{BensonSolomonFSCorrection}.
 In \cite[Definition 2.6]{BensonSolomonFS} Levi and Oliver define
a group of automorphism $\Gamma:=\Gamma\left(q^{n}\right)\le\Aut\left(S_{0}\right)$
and describe the Benson-Solomon fusion system over $S$ as 
\begin{equation}
\F_{\Sol}:=\F_{\Sol}\left(q^{n}\right)=\left\langle \FAB{S}{\Spin_{7}\left(q^{n}\right)},\FAB{S_{0}}{\Gamma}\right\rangle .\label{eq:benson-solomon-description}
\end{equation}
That is $\F_{\Sol}$ is the fusion system over $S$ whose morphism
can be written as compositions of finitely many morphism that are
given either by conjugation with an element in $\Spin_{7}\left(q^{n}\right)$
or by restricting a morphism in $\Gamma$.

Applying Alperin's fusion theorem we can deduce that $S_{0}$ is $\F$-centric
and, since $S_{0}$ has index $2$ in $S$, we can also deduce that
$S_{0}$ is normal in $S$. Therefore we can apply \cite[Proposition C]{SubgroupFamiliesControllingpLocalFiniteGroups}
to conclude that 
\begin{align*}
S & \in\Syl_{2}\left(\Aut_{\L}\left(S_{0}\right)\right) & \text{and} &  & N_{\F_{\Sol}}\left(S_{0}\right) & =\FAB{S}{\Aut_{\L}\left(S_{0}\right)}
\end{align*}
where $\L:=\L\left(q^{n}\right)$ is the centric linking system associated
to $\F_{\Sol}$. With this setup we can rewrite Equation (\ref{eq:benson-solomon-description})
as $\F_{\Sol}=\left\langle \FAB{S}{\Spin_{7}\left(q^{n}\right)},\FAB{S}{\Aut_{\L}\left(S_{0}\right)}\right\rangle $
and we obtain the following.
\begin{prop}
\label{prop:theorem-to-benson-solomon}With notation as above Theorem
\hyperref[thm:A]{A} holds with $p:=2$, $S:=S\left(q^{n}\right)$,
$G_{1}:=\Spin_{7}\left(q^{n}\right)$, $G_{2}:=\Aut_{\L}\left(S_{0}\right)$,
$\F:=\F_{\Sol}$ and $M$ the contravariant part of a Mackey functor
over $\F$ with coefficients in $\Fp[2]$.
\end{prop}

\begin{proof}
From the above discussion we know that $S$ is a Sylow $2$-subgroup
of both $G_{1}$ and $G_{2}$. Since $\F$ is generated by the fusion
subsystems $\F_{1}:=\FSG[G_{1}]$ and $\F_{2}:=\FSG[G_{2}]$ we can
deduce that $\F=\FSG[G_{1}*_{S}G_{2}]$.

From \cite[Proposition 3.3 (a)]{BensonSolomonFS} we know that $\F$-centric
subgroups of $S$ coincide with $\F_{1}$-centric subgroups of $S$.
We can conclude from \cite[Theorem B]{diaz2014mackey} that $\limn_{\OFC[\F_{1}]}\left(M\downarrow_{\OFC[\F_{1}]}^{\OF}\right)=0$
for every $n\ge1$.

Since $\F_{2}:=N_{\F}\left(S_{0}\right)$ and $S_{0}$ is $\F$-centric
we can apply \cite[Lemma 10.4]{HigherLimitsOverTheFusionOrbitCategory}
to deduce that every $\F_{2}$-centric-radical subgroup of $S$ is
$\F$-centric. Since $\F_{2}\subseteq\F$ we can also deduce that
every $\F$-centric subgroup of $S$ is $\F_{2}$-centric. We can
therefore apply \cite[Proposition 10.5]{HigherLimitsOverTheFusionOrbitCategory}
and Corollary \ref{cor:iso-on-higher-limits-Zp.} to deduce that $\limn_{\OFC[\F_{2}]}\left(M\downarrow_{\OFC[\F_{2}]}^{\OF}\right)\cong\limn_{\OFc[\F_{2}]}\left(M\downarrow_{\OFc[\F_{2}]}^{\OF}\right)$
for every $n$. Applying \cite[Theorem B]{diaz2014mackey} we conclude
that $\limn_{\OFC[\F_{2}]}\left(M\downarrow_{\OFC[\F_{2}]}^{\OF}\right)=0$
for every $n\ge1$.

Finally, since the only $\FS$-centric-radical subgroup of $S$ is
$S$, we can apply the same reasoning as before to conclude that $\limn_{\OFC[\FS]}\left(M\downarrow_{\OFC[\FS]}^{\OF}\right)=0$
for every $n\ge1$. All the conditions needed to apply Theorem \hyperref[thm:A]{A}
are therefore satisfied with the setup in the statement. This concludes
the proof.
\end{proof}
Proposition \ref{prop:theorem-to-benson-solomon} allows us to apply
Theorem \hyperref[thm:A]{A} to the family of Benson-Solomon fusion
systems. However the problem still remains that the functor $\CGpC$
might be too complicated to work with. A step towards simplifying
this problem is provided by Aschbacher and Chermak.

Let $q\equiv\pm3\text{ mod }5$ be a prime. In \cite{GroupTheoreticApproachFamilyOf2LocalFiniteGroups}
Aschbacher and Chermak construct for each such $q$ groups $H,B,K$
satisfying $B\le H,K$. They then define the amalgam $G:=H*_{B}K$
and describe for every $n\ge1$ an automorphism $\sigma:=\psi_{n}$
of $G$ such that $\sigma$ restricts to automorphisms of $H,B$ and
$K$ (see \cite[Lemma 5.7]{GroupTheoreticApproachFamilyOf2LocalFiniteGroups}).
Denoting by $X_{\sigma}$ the subgroup of $X\in\left\{ H,B,K,G\right\} $
fixed under the automorphism $\sigma$ they prove in \cite[Lemma 7.4(b)]{GroupTheoreticApproachFamilyOf2LocalFiniteGroups}
that there exits an isomorphism $G_{\sigma}\cong H_{\sigma}*_{B_{\sigma}}K_{\sigma}$.
In \cite[Theorem A(3)]{GroupTheoreticApproachFamilyOf2LocalFiniteGroups}
they use this construction to prove that there exists $S_{\sigma}\in\Syl_{2}\left(H_{\sigma}\right)$
satisfying 
\[
\F_{\Sol}:=\F_{\Sol}\left(q^{2^{n}}\right)=\FAB{S_{\sigma}}{G_{\sigma}}.
\]
Finally, in \cite[Theorem A(4)]{GroupTheoreticApproachFamilyOf2LocalFiniteGroups},
they provide a mapping $\theta_{\sigma}:\OFc[\F_{\Sol}]^{\text{op}}\to\text{Grp}$
which sends every $P\in\F_{\Sol}^{c}$ to a complement $\theta_{\sigma}\left(P\right)$
of $Z\left(P\right)$ in $C_{G_{\sigma}}\left(P\right)$.

Given the description of $\theta_{\sigma}$ provided in \cite{GroupTheoreticApproachFamilyOf2LocalFiniteGroups}
we believe that the following, stated as Theorem \hyperref[thm:B]{B}
in Section \ref{sec:Introduction.}, might prove useful as a means
of studying the sharpness conjecture over certain Benson-Solomon fusion
systems.
\begin{thm}
\label{thm:relation-to-signalizer-functor}With notation as above
we have that:
\begin{enumerate}
\item \label{enu:Theorem-B-1}Theorem \hyperref[thm:A]{A} holds with $p:=2$,
$S:=S_{\sigma}\in\Syl_{2}\left(\Spin_{7}\left(q^{2^{n}}\right)\right)$,
$G_{1}:=H_{\sigma}$, $G_{2}:=K_{\sigma}$, $\F:=\F_{\Sol}$ and $M$
the contravariant part of a Mackey functor over $\F$ with coefficients
in $\Fp[2]$.
\item \label{enu:Theorem-B-2}For every $\F_{\Sol}$-centric subgroup $P$
of $S$ the group $\CGpC\left(P\right)$ is the tensor product of
$\Fp\left[2\right]$ with the abelianization of an extension of $\theta_{\sigma}\left(P\right)$.
\end{enumerate}
\end{thm}

\begin{proof}
We know from \cite[(4.2.2)]{GroupTheoreticApproachFamilyOf2LocalFiniteGroups}
that $H_{\sigma}\cong\Spin_{7}\left(q^{2^{n}}\right)$. In particular
$H_{\sigma}$ is finite.

From \cite[Section 5]{GroupTheoreticApproachFamilyOf2LocalFiniteGroups}
we also know that the group $K$ is isomorphic to the semidirect product
$K\cong B_{0}\rtimes S_{3}$ for some $B_{0}\le H$.  Viewing $\sigma$
as an automorphism of $B_{0}\rtimes S_{3}$ via this isomorphism we
know from \cite[Lemma 5.7]{GroupTheoreticApproachFamilyOf2LocalFiniteGroups}
that for every $\left(x,\varphi\right)\in B_{0}\rtimes S_{3}$ then
$\sigma\left(\left(x,\varphi\right)\right)=\left(\sigma\left(x\right),\varphi\right)$.
In particular $K_{\sigma}\cong\left(B_{0}\right)_{\sigma}\rtimes S_{3}$.
Since $\left(B_{0}\right)_{\sigma}\le H_{\sigma}$ we conclude that
$K_{\sigma}$ is also finite. 

From \cite[Lemma 7.5(a)]{GroupTheoreticApproachFamilyOf2LocalFiniteGroups}
we know that $S_{\sigma}$ is a Sylow $2$-subgroup of, $B_{\sigma}$,
$H_{\sigma}$ and $K_{\sigma}$.

From \cite[Theorem A(3)]{GroupTheoreticApproachFamilyOf2LocalFiniteGroups}
we know that $\F_{\Sol}=\FAB{S_{\sigma}}{G_{\sigma}}$. Since $G_{\sigma}\cong H_{\sigma}*_{B_{\sigma}}K_{\sigma}$
(see \cite[Lemma 7.4(b)]{GroupTheoreticApproachFamilyOf2LocalFiniteGroups})
we can deduce that $\F_{\Sol}$ is generated by the fusion subsystems
$\F_{1}:=\FAB{S_{\sigma}}{H_{\sigma}}$ and $\F_{2}:=\FAB{S_{\sigma}}{K_{\sigma}}$.
We can therefore conclude that $\F_{\Sol}=\FSG[H_{\sigma}*_{S}K_{\sigma}]$.

From \cite[Proposition 3.3 (a)]{BensonSolomonFS} we know that the
$\F$-centric subgroups of $S$ coincide with the $\F_{1}$-centric
subgroups of $S$. We conclude from \cite[Theorem B]{diaz2014mackey}
that $\limn_{\OFC[\F_{1}]}\left(M\downarrow_{\OFC[\F_{1}]}^{\OF}\right)=0$
for every $n\ge0$.

Let's now prove that the same happens for $\F_{2}$. With notation
as above we know from \cite[Lemma 4.5]{GroupTheoreticApproachFamilyOf2LocalFiniteGroups}
that there exists a normal subgroup $U\trianglelefteq S_{\sigma}$
such that $U\cong C_{2}\times C_{2}$ and that $\left(B_{0}\right)_{\sigma}=C_{H_{\sigma}}\left(U\right)$.
Since $U$ is normal in $S_{\sigma}$ we deduce that $C_{S_{\sigma}}\left(U\right)=\left(B_{0}\right)_{\sigma}\cap S_{\sigma}$
has index at most $2$ in $S_{\sigma}$. Since the center of $S_{\sigma}$
has order $p$ we conclude that $C_{S_{\sigma}}\left(U\right)$ has
index $2$ in $S_{\sigma}$ and that $U$ is fully $\F_{1}$-centralized.
Since $U$ is abelian we know from \cite[Proposition 2.5(a)]{BrotoLeviOliverHomotopyTheoryOfFusionSystems}
that a subgroup of $C_{S_{\sigma}}\left(U\right)$ is $\F_{1}$-centric
if and only if it is $\FAB{C_{S_{\sigma}}\left(U\right)}{\left(B_{0}\right)_{\sigma}}$-centric.
Because of \cite[Proposition 3.3 (a)]{BensonSolomonFS} this is equivalent
to saying that a subgroup of $C_{S_{\sigma}}\left(U\right)$ is $\F_{\Sol}$-centric
if and only if it is $\FAB{C_{S_{\sigma}}\left(U\right)}{\left(B_{0}\right)_{\sigma}}$-centric.
Since $\FAB{C_{S_{\sigma}}\left(U\right)}{\left(B_{0}\right)_{\sigma}}\subseteq\F_{2}\subseteq\F_{\Sol}$
we can conclude that every subgroup of $S_{\sigma}$ that is $\F_{\Sol}$-conjugate
to a subgroup of $C_{S_{\sigma}}\left(U\right)$ is $\F_{2}$-centric
if and only if it is $\F_{\Sol}$-centric. On the other hand we know
from \cite[Lemma 1.10]{GroupTheoreticApproachFamilyOf2LocalFiniteGroups}
that $\F_{2}$ is generated by the fusion subsystems $\FS[S_{\sigma}]$
and $\FAB{C_{S_{\sigma}}\left(U\right)}{K_{\sigma}}$. Therefore,
for every $P\le S_{\sigma}$ that is not $\F_{2}$-conjugate to a
subgroup of $C_{S_{\sigma}}\left(U\right)$, we have $\Aut_{\F_{2}}\left(P\right)=\Aut_{\FS[S_{\sigma}]}\left(P\right)$.
We conclude that the only $\F_{2}$-centric-radical subgroup of $S_{\sigma}$
that is not $\F_{2}$-conjugate to a subgroup of $C_{S_{\sigma}}\left(U\right)$
is $S_{\sigma}$. Since $S_{\sigma}$ is $\F_{\Sol}$-centric then
we can apply \cite[Proposition 10.5]{HigherLimitsOverTheFusionOrbitCategory}
and Corollary \ref{cor:iso-on-higher-limits-Zp.} to deduce that $\limn_{\OFC[\F_{2}]}\left(M\downarrow_{\OFC[\F_{2}]}^{\OF}\right)\cong\limn_{\OFc[\F_{2}]}\left(M\downarrow_{\OFc[\F_{2}]}^{\OF}\right)$
for every $n$. We conclude from \cite[Theorem B]{diaz2014mackey}
that $\limn_{\OFC[\F_{2}]}\left(M\downarrow_{\OFC[\F_{2}]}^{\OF}\right)=0$
for every $n\ge1$. 

Using the same arguments we deduce that $\limn_{\OFC[\F_{S_{\sigma}}\left(S_{\sigma}\right)]}\left(M\downarrow_{\OFC[\F_{S_{\sigma}}\left(S_{\sigma}\right)]}^{\OF}\right)=0$
for every $n\ge0$. All conditions needed to apply Theorem \hyperref[thm:A]{A}
are therefore satisfied thus proving Item (\ref{enu:Theorem-B-1}).

\medskip{}

Let $\hat{G}_{\sigma}:=H_{\sigma}*_{S_{\sigma}}K_{\sigma}$, let $P\le S$
and let $\pi:\hat{G}_{\sigma}\to G_{\sigma}$ be the natural surjective
map. Let $x=a_{H_{\sigma},n}a_{K_{\sigma},n}\cdots a_{H_{\sigma},1}a_{K_{\sigma},1}$
be an element of $C_{G_{\sigma}}\left(P\right)$ with $a_{J,i}\in J$
for every $J\in\left\{ H_{\sigma},K_{\sigma}\right\} $ and $i=1,\dots,n$.
Since $\lui{x}{P}=P\le B_{\sigma}$ then \cite[Lemma 1]{RobinsonAmalgams}
tells us that for every $i=1,\dots,n$ we have that $\hat{a}_{H_{\sigma},i}:=a_{H_{\sigma},i}\cdots a_{K_{\sigma},1}\in N_{G_{\sigma}}\left(P,B_{\sigma}\right)$
and that $\hat{a}_{K_{\sigma},i}:=a_{K_{\sigma},i}\cdots a_{K_{\sigma},1}\in N_{G_{\sigma}}\left(P,B_{\sigma}\right)$.
Since $S_{\sigma}\in\Syl_{2}\left(B_{\sigma}\right)$ and $P$ is
a $p$-group then for every $i=1,\dots,n$ there exist elements $b_{K_{\sigma},i},b_{H_{\sigma},i}\in B_{\sigma}$
such that $b_{K_{\sigma},i}\hat{a}_{K_{\sigma},i}$ and $b_{H_{\sigma},i}\hat{a}_{H_{\sigma},i}$
are elements of $N_{G_{\sigma}}\left(P,S_{\sigma}\right)$. Without
loss of generality we can take $b_{H_{\sigma},n}$ to be the identity
element and define $b_{H_{\sigma},0}:=b_{H_{\sigma},n}$. For every
$i=1,\dots,n$ define $c_{K_{\sigma},i}:=b_{K_{\sigma},i}a_{K_{\sigma},i}b_{H_{\sigma},i-1}^{-1}$
and $c_{H_{\sigma},i}:=b_{H_{\sigma},i}a_{H_{\sigma},i}b_{K_{\sigma},i}^{-1}$.
Viewing $c_{K_{\sigma},i}$ and $c_{H_{\sigma},i}$ as elements of
$\hat{G}_{\sigma}$ we can now define $\tilde{x}=c_{H_{\sigma},n}c_{K_{\sigma},n}\cdots c_{H_{\sigma},1}c_{K_{\sigma},1}\in\hat{G}_{\sigma}$.
By construction we have that $x=\pi\left(\tilde{x}\right)$ and that
both $c_{H_{\sigma},i}\cdots c_{K_{\sigma},1}$ and $c_{K_{\sigma},i}\cdots c_{K_{\sigma},1}$
are elements of $N_{\hat{G}_{\sigma}}\left(P,S\right)$ for every
$i=1,\dots,n$. Since $x\in C_{G_{\sigma}}\left(P\right)$ then we
deduce that $\tilde{x}\in C_{\hat{G}_{\sigma}}\left(P\right)$. We
conclude that the restriction $\pi_{|C_{\hat{G}_{\sigma}}\left(P\right)}$
of $\pi$ to the centralizer in $\hat{G}_{\sigma}$ of $P$ is also
a surjective group morphism onto $C_{G_{\sigma}}\left(P\right)$.
This morphism leads in turn to a surjective group morphism $\tilde{\pi}:C_{\hat{G}_{\sigma}}\left(P\right)/Z\left(p\right)\twoheadrightarrow C_{G_{\sigma}}\left(P\right)/Z\left(p\right)$.
Since $C_{G_{\sigma}}\left(P\right)\cong Z\left(p\right)\times\theta_{\sigma}\left(P\right)$
then the result follows.
\end{proof}

\appendix

\section{\protect\label{app:Higher-limits-over-Fp-and-Zp.}Higher limits over
$\Fp$-modules and $\Zp$-modules.}

Let $p$ be a prime, let $\boldsymbol{C}$ be a small category with
finitely many objects, let $F:\boldsymbol{C}\to\Fp\text{-Mod}$ be
a functor and let $\iota:\Fp\text{-Mod}\to\Zp\text{-Mod}$ be the
natural inclusion of categories. In this appendix we prove that $\limn_{\boldsymbol{C}}\left(F\right)\cong\limn_{\boldsymbol{C}}\left(\iota\circ F\right)$
as abelian groups for every non negative integer $n$. This results
is widely used in the literature, but we were unable to find any reference
proving it. Therefore, for the sake of completeness, we include a
proof in this paper.

Let us start by relating free resolutions of $\Zp$-modules with free
resolutions of $\Fp$-modules.
\begin{lem}
\label{lem:free-res-to-free-res.}Let $\boldsymbol{C}$ be a small
category with finitely many objects, let $\R$ be a commutative ring,
let $x\in\R$ be a non zero divisor, let $M$ be an $\R\boldsymbol{C}$-module
such that $m\cdot x\not=0$ for every $m\in M$, let $F_{*}\stackrel{\varepsilon}{\to}M\to0$
be a free resolution of $M$ in $\R\boldsymbol{C}$ and define the
quotient ring $\RR:=\R/x\R$. By viewing $\RR\boldsymbol{C}$ as an
$\left(\R\boldsymbol{C},\RR\boldsymbol{C}\right)$-bimodule in the
natural way we have that $F_{*}\otimes_{\R\boldsymbol{C}}\RR\boldsymbol{C}\stackrel{\varepsilon\otimes\Id_{\RR\boldsymbol{C}}}{\longrightarrow}M\otimes_{\R\boldsymbol{C}}\RR\boldsymbol{C}\to0$
is a free resolution of the $\RR\boldsymbol{C}$-module $M\otimes_{\R\boldsymbol{C}}\RR\boldsymbol{C}$.
\end{lem}

\begin{proof}
For every free $\R\boldsymbol{C}$-module $F\cong\left(\R\boldsymbol{C}\right)^{n}$
we know that $F\otimes_{\R\boldsymbol{C}}\RR\boldsymbol{C}\cong\left(\RR\boldsymbol{C}\right)^{n}$
is a free $\RR\boldsymbol{C}$-module. Therefore we only need to prove
that the sequence in the statement is exact. 

For every $\R\boldsymbol{C}$-module $N$ let $x\cdot:N\to N$ denote
the endomorphism of $\R\boldsymbol{C}$-modules corresponding to multiplication
by $x$ (seen as an element in $\R\boldsymbol{C}$). Since $x$ is
not a zero divisor of $\R$ then, when seen as an element in $\R\boldsymbol{C}$,
it is not a zero divisor of $\R\boldsymbol{C}$ either. Therefore,
since each $F_{i}$ is free we can deduce that the endomorphisms $x\cdot:F_{i}\to F_{i}$
are injective for every $i$. By definition of $M$ we also know that
the endomorphism $x\cdot:M\to M$ is injective. Take now the chain
complex $C_{*}$ defined by setting $C_{-1}=M$, $C_{i}=F_{i}$ for
every $i\ge0$ and $C_{j}=0$ for every $j\le-2$. Since $\RR\boldsymbol{C}\cong\R\boldsymbol{C}/\left(\R\boldsymbol{C}x\right)$
by construction then, viewing $\RR\boldsymbol{C}$ as an $\left(\R\boldsymbol{C},\R\boldsymbol{C}\right)$-bimodule
in the natural way, we obtain the exact sequence of chain complexes
\[
0\to C_{*}\stackrel{x\cdot}{\to}C_{*}\to\left(C_{*}\otimes_{\R\boldsymbol{C}}\RR\boldsymbol{C}\right)\to0.
\]
This in turn leads (see \cite[Theorem 1.3.1]{WeibelIntroductionToHomologicalAlgebra1994})
to the following long exact sequence in homology 
\[
\cdots\to H_{1}\left(C_{*}\otimes_{\R\boldsymbol{C}}\RR\boldsymbol{C}\right)\to H_{0}\left(C_{*}\right)\to H_{0}\left(C_{*}\right)\to H_{0}\left(C_{*}\otimes_{\R\boldsymbol{C}}\RR\boldsymbol{C}\right)\to0.
\]
Since $F_{*}\stackrel{\varepsilon}{\to}M\to0$ is a free resolution
of $M$ then, in particular, we have that it is a long exact sequence.
Equivalently we have that $H_{n}\left(C_{*}\right)=0$ for every $n\in\Z$.
From the above long exact sequence we can therefore conclude that
$H_{n}\left(C_{*}\otimes_{\R\boldsymbol{C}}\RR\boldsymbol{C}\right)=0$
for every $n\in\Z$. Equivalently the sequence of $\R\boldsymbol{C}$-modules
$F_{*}\otimes_{\R\boldsymbol{C}}\RR\boldsymbol{C}\stackrel{\varepsilon\otimes\Id_{\RR\boldsymbol{C}}}{\longrightarrow}M\otimes_{\R\boldsymbol{C}}\RR\boldsymbol{C}$
is exact. The result follows from viewing this sequence as a sequence
of $\RR\boldsymbol{C}$-modules.
\end{proof}
The following result, together with Lemma \ref{lem:free-res-to-free-res.},
allows us to switch from free resolutions in $\Zp\Mod$ to free resolutions
in $\Fp\Mod$ by applying a certain contravariant $\Hom$ functor.
\begin{lem}
\label{lem:isomorphism-of-homs.}Let $\R$ be a ring, let $I$ be
a two sided ideal of $\R$, let $\RR:=\R/I$, let $\iota:\RR\Mod\hookrightarrow\R\Mod$
be the natural inclusion of categories, let $M$ be an $\RR$-module
and let $N$ be an $\R$-module. Viewing $\RR$ as an $\left(\R,\RR\right)$-bimodule
in the natural way there exists an isomorphism of abelian groups $\Hom_{\R}\left(N,\iota\left(M\right)\right)\cong\Hom_{\RR}\left(N\otimes_{\R}\RR,M\right)$
which is natural in both $N$ and $M$.
\end{lem}

\begin{proof}
The functor $\iota$ is in fact the restriction functor resulting
from the projection $\pi:\R\to\RR$. It is well known that such restriction
functor is right adjoint to the functor $-\otimes_{\R}\RR:\R\Mod\to\RR\Mod$.
More precisely there exists a natural bijection $\Gamma:\Hom_{\R}\left(N,\iota\left(M\right)\right)\bjarrow\Hom_{\RR}\left(N\otimes_{\R}\RR,M\right)$
which sends every morphism $f\in\Hom_{\R}\left(N,\iota\left(M\right)\right)$
to the morphism $\Gamma\left(f\right)\in\Hom_{\RR}\left(N\otimes_{\R}\RR,M\right)$
defined by setting 
\[
\Gamma\left(f\right)\left(n\otimes\pi\left(x\right)\right)=f\left(n\right)\cdot\pi\left(x\right),
\]
for every $n\in N$ and every $x\in\R$.

For every $f,g\in\Hom_{\R}\left(N,\iota\left(M\right)\right)$, every
$x\in\R$ and every $n\in N$ we have that
\[
\Gamma\left(f+g\right)\left(n\otimes\pi\left(x\right)\right)=\left(f+g\right)\left(n\right)\cdot\pi\left(x\right)=f\left(n\right)\cdot\pi\left(x\right)+g\left(n\right)\cdot\pi\left(x\right)=\Gamma\left(f\right)\left(n\otimes\pi\left(x\right)\right)+\Gamma\left(g\right)\left(n\otimes\pi\left(x\right)\right).
\]
We can therefore deduce that $\Gamma\left(f+g\right)=\Gamma\left(f\right)+\Gamma\left(g\right)$.
Since $\Gamma$ sends the zero morphism to the zero morphism this
implies that $\Gamma$ is in fact a morphism of abelian groups thus
concluding the proof.
\end{proof}
As a consequence of Lemmas \ref{lem:free-res-to-free-res.} and \ref{lem:isomorphism-of-homs.}
we have the following.
\begin{prop}
\label{prop:iso-on-higher-limits-general.}Let $\R,\boldsymbol{C},x$
and $\RR$ be as in Lemma \ref{lem:free-res-to-free-res.} and let
$\iota:\RR\Mod\to\R\Mod$ be the natural inclusion of categories.
For every non negative integer $n$ and every contravariant functor
$M:\boldsymbol{C}^{\text{op}}\to\RR\Mod$ there exists an isomorphism
of abelian groups $\limn_{\boldsymbol{C}}\left(M\right)\cong\limn_{\boldsymbol{C}}\left(\iota\circ M\right)$
which is natural in $M$.
\end{prop}

\begin{proof}
Let $\underline{\R}^{\boldsymbol{C}}$ be as in Definition \ref{def:Constant-functor.},
let $F_{*}\stackrel{\varepsilon}{\to}\underline{\R}^{\boldsymbol{C}}\to0$
be a free resolution of $\underline{\R}^{\boldsymbol{C}}$ in $\R\boldsymbol{C}\Mod$,
denote by $d_{n}:F_{n+1}\to F_{n}$ its differentials and view $\iota\circ M:\boldsymbol{C}^{\text{op}}\to\R\Mod$
as an $\R\boldsymbol{C}$-module (see \cite[Proposition 2.1]{WebbRepresentationAndCohomologyOfCategories}).
By definition of the $\Ext_{\R\boldsymbol{C}}^{n}$ groups, for every
integer $n\ge0$ we have the isomorphism of abelian groups
\[
\Ext_{\R\boldsymbol{C}}^{n}\left(\underline{\R}^{\boldsymbol{C}},\iota\circ M\right)\cong\frac{\ker\left(d_{n}^{*}:\Hom_{\R}\left(F_{n},\iota\circ M\right)\to\Hom_{\R}\left(F_{n+1},\iota\circ M\right)\right)}{\im\left(d_{n-1}^{*}:\Hom_{\R}\left(F_{n-1},\iota\circ M\right)\to\Hom_{\R}\left(F_{n},\iota\circ M\right)\right)},
\]
where we take $d_{-1}:=0$ and $F_{-1}=0$. Because of Lemma \ref{lem:isomorphism-of-homs.}
we can obtain from the above the following isomorphism of abelian
groups
\begin{equation}
\Ext_{\R\boldsymbol{C}}^{n}\left(\underline{\R}^{\boldsymbol{C}},\iota\circ M\right)\cong\frac{\ker\left(\left(d_{n}\otimes\Id_{\RR\boldsymbol{C}}\right)^{*}:\Hom_{\RR}\left(F_{n}\otimes_{\R\boldsymbol{C}}\RR\boldsymbol{C},M\right)\to\Hom_{\RR}\left(F_{n+1}\otimes_{\R\boldsymbol{C}}\RR\boldsymbol{C},M\right)\right)}{\im\left(\left(d_{n-1}\otimes\Id_{\RR\boldsymbol{C}}\right)^{*}:\Hom_{\RR}\left(F_{n-1}\otimes_{\R\boldsymbol{C}}\RR\boldsymbol{C},M\right)\to\Hom_{\RR}\left(F_{n}\otimes_{\R\boldsymbol{C}}\RR\boldsymbol{C},M\right)\right)}.\label{eq:ext-Zp-and-Ext-fp.}
\end{equation}
Since $\underline{\R}^{\boldsymbol{C}}\otimes_{\R\boldsymbol{C}}\RR\boldsymbol{C}\cong\underline{\RR}^{\boldsymbol{C}}$
as $\RR\boldsymbol{C}$-modules then, from Lemma \ref{lem:free-res-to-free-res.}
we obtain the free resolution $F_{*}\otimes_{\R\boldsymbol{C}}\RR\boldsymbol{C}\stackrel{\varepsilon\otimes\Id_{\RR\boldsymbol{C}}}{\longrightarrow}\underline{\RR}^{\boldsymbol{C}}\to0$
of $\underline{\RR}^{\boldsymbol{C}}$ in $\RR\boldsymbol{C}\Mod$.
Therefore, from definition of the $\Ext_{\RR\boldsymbol{C}}^{n}$
groups, we can rewrite Equation (\ref{eq:ext-Zp-and-Ext-fp.}) as
\[
\Ext_{\R\boldsymbol{C}}^{n}\left(\underline{\R}^{\boldsymbol{C}},\iota\circ M\right)\cong\Ext_{\RR\boldsymbol{C}}^{n}\left(\underline{\RR}^{\boldsymbol{C}},M\right).
\]
The result follows from the above equivalence of abelian groups and
\cite[Corollary 5.2]{WebbRepresentationAndCohomologyOfCategories}.
\end{proof}
As a corollary of Proposition \ref{prop:iso-on-higher-limits-general.}
we obtain the result that motivates the introduction of this appendix.
\begin{cor}
\label{cor:iso-on-higher-limits-Zp.}Let $\boldsymbol{C}$ be a small
category with finitely many objects, let $p$ be a prime and let $\iota:\Fp\Mod\to\Zp\Mod$
be the natural inclusion. Then, for every non negative integer $n$
and every contravariant functor $M:\boldsymbol{C}^{\text{op}}\to\Fp\Mod$
there exists an isomorphism of abelian groups $\limn_{\boldsymbol{C}}\left(M\right)\cong\limn_{\boldsymbol{C}}\left(\iota\circ M\right)$
which is natural in $M$.
\end{cor}

\begin{proof}
This is just a particular case of Proposition \ref{prop:iso-on-higher-limits-general.}
taken with $\R:=\Zp$ and $x:=p$.
\end{proof}
\printbibliography

\end{document}